\documentclass[10pt]{article}
\usepackage[utf8]{inputenc}
\usepackage[a4paper,margin=1.2in]{geometry}

\usepackage{graphicx}
\usepackage{amsmath}
\usepackage{ dsfont }
\usepackage{amsfonts}
\usepackage{amsthm}
\usepackage{ amssymb }
\usepackage{xcolor}
\usepackage{pgfplots}
\usepackage{subcaption}

\def\p{\partial}

\def\x{\xi}

\def\sub{\subset}

\DeclareMathOperator{\sgn}{sgn}
\DeclareMathOperator{\supp}{supp}
\newcommand{\R}{{\mathbb{R}}}

\newcommand{\Z}{{\mathbb{Z}}}

\newtheorem{df}{Definition}

\newtheorem{lem}{Lemma}
\newtheorem{prop}{Proposition}
\newtheorem{rem}{Remark}

\newtheorem{thm}{Theorem}

\title{Finite time singularities to the 3D incompressible Euler equations for solutions in $C^{\infty}(\mathbb{R}^3 \setminus \{0\})\cap C^{1,\alpha}\cap L^2$}
\author{Diego Cordoba\footnote{dcg@icmat.es},  Luis Martinez-Zoroa\footnote{luis.martinez@icmat.es} and Fan Zheng \footnote{fan.zheng@icmat.es}\\ \\ \small Instituto de Ciencias Matematicas CSIC-UAM-UCM-UC3M }

\pgfplotsset{compat=newest}
\begin{document}

\maketitle

\hfill {\it My thoughts abound, filling all space}

\hfill {\it In silence, a thunder crashes}

\hfill {\it ---Lu Xun}

\begin{abstract}
     We introduce a novel mechanism that reveals finite time singularities within the 1D De Gregorio model and the 3D incompressible Euler equations. Remarkably, we do not construct our blow up using self-similar coordinates, but build it from infinitely many regions with vorticity, separated by vortex-free regions in between. It yields solutions of the 3D incompressible Euler equations in $\mathbb{R}^3\times [-T,0]$ such that the velocity is in the space $C^{\infty}(\mathbb{R}^3 \setminus \{0\})\cap C^{1,\alpha}\cap L^2$ where $0 < \alpha \ll 1$ for times $t\in (-T,0)$ and is not $C^1$ at time 0.
    
\end{abstract}

\section{Introduction}

 We consider the incompressible Euler equations 
\begin{eqnarray}\label{Euler}
\p_t u + (u\cdot\nabla) u + \nabla  p= 0,\\
\nabla\cdot u = 0 \nonumber
\end{eqnarray}
in $\mathbb{R}^3 \times \mathbb{R_+}$, where $u(x,t)=(u_1(x,t), u_2(x,t), u_3(x,t))$ is the velocity field and $p=p(x,t)$ is the pressure function of an ideal, incompressible fluid flow with zero viscosity. In this paper we study the initial value problem for (\ref{Euler}) with a given initial divergence free data $$u_0(x)=u(x,0)\in C^{\infty}(\mathbb{R}^3 \setminus \{0\})\cap C^{1,\alpha}\cap L^2.$$

The classical theory of the Euler equations goes back to the work of Lichtenstein \cite{Lich} and Gunther \cite{Gunther}, who showed local well-posedness for $u_0$ in $C^{k,\alpha}$ ($k\geq 1$, $\alpha\in(0,1)$). An outcome arising from the incompressibility for solutions in $C^1$ (i.e. $\nabla\cdot u =0$) is that the energy $\int_{\mathbb{R}^3}|u(x,t)|^2 dx $ is conserved for all time $t$.
Later, Ebin and Marsden \cite{EM} established local well-posedness in $H^s$ for Sobolev spaces ($s>\frac{3}{2} + 1$) within a compact domain $\Omega \subset \mathbb{R}^3$.  Bourguignon and Brezis \cite{BB}  extended this result to the space $W^{s,p}$ with the condition $s>\frac{3}{p} + 1$. Additionally, Kato \cite{K} further generalized the local well-posedness to $\mathbb{R}^3$ for initial data $u_0$ in $H^s$, where $s>\frac{3}{2} + 1$. This extension to the $W^{s,p}$ spaces was achieved by Kato and Ponce \cite{KP}. This was later extended to various other function spaces by Chae in \cite{Chae1}. 

An outstanding open problem is to determine whether a 3D incompressible Euler smooth solution in $L^2$ (i.e. with finite energy) can exhibit a singularity at a finite time. A classic criterion result of Beale–Kato–Majda \cite{BKM} states that, if a singularity forms at time $T$, then the vorticity $\omega(x, t) = \nabla\times u(x, t)$ grows so rapidly that  $$\int_0^T \|\omega\|_{L^{\infty}}(s) ds = \infty.$$ 
Another very useful criterion for blow-up is the work by Constantin–Fefferman–Majda (sede \cite{CFM}) showed that, if the velocity remain bounded up to the time T of singularity formation, then the vorticity direction $\omega(x, t)/|\omega(x, t)|$ cannot remain uniformly Lipschitz continuous up to time T. Many additional blow-up criteria have been established for the Euler equations in subsequent studies. For further details, please refer to \cite{Chae2} and the accompanying references. In the context of two-dimensional scenarios, the aforementioned local well-posedness can be effortlessly extended to encompass all time intervals by employing the Beale--Kato--Majda criterion since the flow itself carries the vorticity. However, it is important to note that singular behavior may arise. For instance, Kiselev and Šverák \cite{KS} achieved the optimal growth bound for the 2D incompressible Euler within a disk. Additionally, in a recent publication \cite{CMZ}, two of the authors of this paper in collaboration with Ozanski proved the existence of global classical solutions in 2D that exhibit an instantaneous gap loss of super-critical Sobolev regularity. For a more comprehensive understanding of the history of the finite time singularity problem for the Euler equation, we suggest referring to the following sources: \cite{BT,Con,Fef,Gib,Kis,MB}. These references provide in-depth reviews on the subject matter.

In a recent work, Elgindi \cite{Elgindi2} proved singularity formation of the 3D Euler equations (\ref{Euler}) with axi-symmetric symmetry and without swirl for $C^{1,\alpha}$ velocity, where $\alpha >0$ is sufficiently small. Additionally, Elgindi, Ghoul, and Masmoudi \cite{Elgindi3} established the stability of the blowup solutions presented in \cite{Elgindi2}. This stability analysis ultimately leads to the construction of finite time singularities of $C^{1,\alpha}$ solutions with finite energy. These solutions in \cite{Elgindi2} and \cite{Elgindi3} have an asymptotically self-similar blow-up profile with non-smooth points on part of the $z$-axis and the $z = 0$ plane. Significant progress has been made recently for self-similar singularities of axi-symmetric flows in the presence of a boundary. We refer the reader first  to the work of Elgindi and Jeong \cite{EJ} where they prove finite-time blow-up solutions in scale-invariant Holder spaces  in domains with corners. Furthermore, Chen and Hou in \cite{Hou} provided a rigorous proof  of nearly self-similar $C^{1,\alpha}$ blow up near a smooth boundary. Later in their subsequent remarkable work \cite{Hou2}  they proved the blow-up of smooth self similar solutions through the use of computer assisted proofs. After the publication of the first draft of this paper, Chen \cite{Chen2} showed that blow up from initial data that is singular at only one point (with and without bouundary) can also be achieved in the self-simiar setting. Additionally, in \cite{WLGB} by Wang--Lai--G\'omez-Serrano--Buckmaster, physics-informed neural networks were employed to construct approximate self-similar blow-up profiles.

It is interesting to note that in Section 10 of \cite{Elgindi2}, Elgindi proposed several open questions aimed at minimizing the presence of the singularities in the initial data in a blow up. Specifically, he remarked that if there is no boundary, one can hope to construct a blow up from initial data that is not smooth only on the $z$ axis, while in the presence of a boundary, a blow up from initial data that is not smooth at a single point would be very interesting. Our result in this paper does more than answering the two questions affirmatively with a pen-paper proof: it shows a blow up in the absence of boundary, from initial data that is not smooth only at the origin. Thanks to the odd symmetry of the velocity field, the $z = 0$ plane can serve as a natural boundary. In the no swirl (and no boundary) setting, our construction has the least number of singularities in the initial data, because global regularity is known for smooth initial data (Section 4.3 of \cite{MB}).

The main goal of this paper is to pursue a new mechanism for a blow-up in $\mathbb{R}^3$, different from a self-similar profile, which will allow to treat a wider space of smooth solutions. In particular we construct solutions of the 3D incompressible Euler equations in $\mathbb{R}^3\times [-T,0]$ such that the velocity is in the space $C^{\infty}(\mathbb{R}^3 \setminus \{0\})\cap C^{1,\alpha}\cap L^2$ where $0<\alpha\ll1$ for times $t\in (-T,0)$ and develops a finite time singularity at time 0. These solutions have the axi-symmetric symmetry without swirl but they are not based on asymptotically self-similar profiles. Instead they consist of infinitely many regions with vorticity, each one of which is much closer to the origin than the previous one. The vorticity is set up in such a way that it generates a hyperbolic saddle at the origin that moves (and deforms) the inner vortices in a way depending only on the strength of the vortex. Thus the dynamics of the Euler equation in this setting can be approximated by an infinite system of ordinary differential equations, which happens to be explicitly solvable. Working backwards from the time of blow up, this allows us to precisely control of the dynamics and show the blow up. It is also worth noting that the same mechanism works for the generalized 1D Constantin--Lax--Majda/De Gregorio model, which we now introduce.

\subsection{Generalization of the Constantin--Lax--Majda/De Gregorio 1D model of the 3D Euler equations}
There is a wide interest in studying 1D models that captures the non-localities and non-linearities of Euler equations in order to obtain insights of how the singularity in 3D can develop. Here we will focus in a family of models that reflects the dynamics of the vorticity. The 3D incompressible Euler equations in vorticity formulation has the following form:
\begin{eqnarray}\label{Eulervorticity}
\p_t \omega + (u\cdot\nabla) \omega = \omega\nabla u,
\end{eqnarray}
where u is determined by $\nabla\times u= \omega$ and $\nabla\cdot u=0$. Thus, a smooth velocity with sufficient decay at infinity can be written in terms of an integral operator with respect to $\omega$ which is known as  the Biot-Savart law:
\begin{eqnarray}\label{Biot}
u(x,t) = \frac{1}{4\pi}\int_{\mathbb{R}^3} \frac{(x-y)\times \omega(y)}{|x-y|^3} dy.
\end{eqnarray}
Moreover, the $\nabla u$ are single integral operators that can be written as a convolution of the vorticity with a homogeneous kernel of order $-3$ and with zero mean on the unit sphere whose one dimensional analogue
is the Hilbert transform. In order to construct lower dimensional models containing some
of the main features of (\ref{Eulervorticity}), Constantin, Lax and Majda \cite{CLM} considered the
scalar equation $$\omega_t = \omega H\omega $$ where $\omega: \mathbb{R} \longrightarrow \mathbb{R}$ and $H\omega$ is the Hilbert transform of $\omega$. They show that this model is solvable with an explicit formula which leads to examples of finite time singularities. The next level of difficulty  is to add an advective term: $$\omega_t + u \omega_x = \omega H\omega.$$ De Gregorio, in \cite{De2}, proposed a velocity given by an integral operator $u(x,t)= \int_0^x H(\omega)(s,t)ds$. Later, in \cite{CCF}, considered a velocity field $u(x,t)= -\int_0^x H(\omega)(s,t)ds$ motivated by studying the Surface Quasi-geostrophic (SQG) equation.  In this case it was shown finite time singularities in \cite{CCF} and \cite{CCF2}. 
A generalization of all those models is given by Okamoto-Sakajo-Wu in \cite{OSW} by taking $u(x,t)= a\int_0^x H(\omega)(s,t)ds$ where $a\in \mathbb{R}$. These family of equations are known as the generalized Constantin-Lax-Majda or by the generalized De Gregorio models.

 Over the past years, significant efforts have been dedicated to the proof and comprehension of singularity formation in this family of 1D models in terms of the parameter $a$. In \cite{CC}, Castro and C\'ordoba established the existence of smooth initial data that give rise to singularities within finite time when $a < 0$. The analysis focuses on the evolution of the quantity $H\omega$, employing a noteworthy inequality $H(Hff_{xx})\geq 0$ for even smooth functions $f$. Additionally, they show non-regular self-similar solutions that exhibit singularities within finite time for positive values of $a$, employing the self-similar profile $\sqrt{1-x^2} $ for $|x|<1$ and 0 for $|x|>1$ . However, it should be noted that the profiles are non-smooth, leaving the question of finite time singularities for smooth initial data and $a > 0$ unresolved. When $a < 0$, the combined effects of advection and vortex stretching work in tandem, leading to a blow-up phenomenon. However, when $a > 0$, these forces counteract each other, resulting in a more intriguing scenario. Castro conducted a preliminary research on $a = 1$, utilizing both analytical and numerical methods in \cite{Cas}. Notably, he proves finite time blow-up arising from $C^{\infty}_c$ initial data, subject to certain convexity and monotonicity assumptions on the solution.

 In their work \cite{EJ2}, Elgindi and Jeong established the existence of initial data in $H^3$ that exhibit self-similar blow-up within finite time, specifically for $a > 0$, provided it is sufficiently small. Remarkably, their proof can be extended to other spaces such as $H^s$, with arbitrary large values of $s$. Building upon this work, Elgindi, Ghoul, and Masmoudi further demonstrated the stability of such singularities in their subsequent research \cite{Elgindi4}. Independently, Chen, Hou, and Huang \cite{Hou3} also established the occurrence of finite time singularities for initial data in the $C_{c}^{\infty}$ space by leveraging self-similar profiles. Importantly, Chen extended in \cite{Chen} their findings to prove the presence of singularities for $C^{\infty}(\mathbb{R}\setminus\{0\})\cap C^{\alpha}$ initial data $a = 1$ and $\alpha\in(0,1)$. Zheng showed recently in \cite{Z} that for any $\alpha\in(0, 1)$ such that $|a\alpha|$ is sufficiently small,
there is an exactly self-similar $C^{\alpha}$ solution that blows up in finite time.
This simultaneously improves on the result in \cite{EJ2} by removing the restriction $\frac{1}{\alpha} \in \Z$ and \cite{Hou3}, which only deals with asymptotically self-similar blow-ups. Very recently, in the notable work \cite{Huang}, the authors constructed self-similar blow ups for all $a \le 1$ that are not singular at the origin, but are either smooth everywhere, or singular only at the endpoints of the support, depending on the range of $a$.

On the other hand if there is enough regularity there is strong evidence for global existence in $S^1$. For more details see \cite{Lei, JSS, Chen}.

In the present work we provide a new mechanism, non self-similar, of blow-up for solutions $\omega$ in $C^{\infty}(\mathbb{R}\setminus\{0\})\cap C^{\alpha}$ with $0<\alpha\ll1$ for all $a\in \mathbb{R}$. The main ideas of this scenario is implemented in the finite time singularities for 3D incompressible Euler constructed in this work.

\subsection{Main results}
The main results of the paper are the following:

\begin{thm}
For any $a > 0$ and $0 < s \ll \min(1, 1/a)$, there are $T > 0$, $C > 0$ and a solution to the generalized De Gregorio equation
\[
\partial_tw + au\partial_xw = wHw, \quad \partial_xu = Hw
\]
on the time interval $-T \le t < 0$ such that for all $t$ in this interval,
the function $w(\cdot, t) \in C^\infty(\R \setminus \{0\}) \cap C^s(\R)$
and that $1/|Ct| \le \max|w(\cdot, t)| \le C/|t|$.
\end{thm}
\begin{rem}
The solution $w$ we construct is odd with respect to the origin, where it vanishes.
Since $\partial_xu = Hw$, $u$ has the same symmetry.
\end{rem}
\begin{rem}
We also show that the rate of blow up is $1/|t|$ (Theorem \ref{rate}) and that the blow up is not asymptotically self-similar (Theorem \ref{not-similar}), as defined in \cite{Chen}.
\end{rem}

\begin{thm}
There exist $0 < \alpha \ll 1$ and solutions of the 3D incompressible Euler equations (\ref{Euler}) in $\mathbb{R}^3\times [-T,0]$ such that on the time interval $-T \le t < 0$, the velocity $u$ is in the space $C^{\infty}(\R^3 \setminus \{0\})\cap C^{1,\alpha}(\R^3)\cap L^2$ and the vorticity $\omega = \nabla \times u$ satisfies $1/|Ct| \le \max|w(\cdot, t)| \le C/|t|$.
\end{thm}

\begin{rem}
The velocity field we construct is an axi-symmetric flow without swirl, whose $z$ component is odd in $z$. Note that smooth axi-symmetric flows without swirl have global regularity, so our construction has the minimal number of singularities in the initial data.
\end{rem}

\begin{rem}
In 1D we can show that our construction yields non-asymptotically self-similar solutions (section \ref{1DNonsim}).
In 3D our solution possesses similar geometry, see section \ref{3DNonsim} for details.
\end{rem}

\begin{rem}
Our solution has finite energy and compactly supported vorticity, whose maximum grows like $1/t$ towards the blow up time.
This is consistent with the Beale--Kato--Majda criterion.
\end{rem}

\begin{rem}
The H\"older exponent of the initial data is effective (i.e., can be computed if one wishes), but is most likely not optimal.
It is an open question whether one can reach the $1/3$ threshold (see Remark 1.5 of \cite{Elgindi2}).
\end{rem}

\begin{figure}
\begin{subfigure}{0.5\textwidth}
    \centering
    \begin{tikzpicture}[xscale=0.25,yscale=0.25]
    \draw[->] (-12,0) -- (12,0) node[below]{$x$} coordinate (x axis);
    \draw[->] (0,-12) -- (0,12) node[right]{$w$} coordinate (y axis);
    \draw[domain=-12:12,samples=400,smooth] plot(\x,
        {12*exp(-81*(\x+1)*(\x+1))-12*exp(-81*(\x-1)*(\x-1))
        +8.8*exp(-9*(\x+5)*(\x+5))-8.8*exp(-9*(\x-5)*(\x-5))
        +8*exp(-(\x+9)*(\x+9))-8*exp(-(\x-9)*(\x-9))});
    \filldraw (-3.6,6) circle [radius=0.1];
    \filldraw (-2.8,6) circle [radius=0.1];
    \filldraw (-2.0,6) circle [radius=0.1];
    \draw[->] (0.5,1.6) -- (1.5,1.6);
    \draw[->] (2,1.6) -- (4,1.6);
    \draw[->] (5,1.6) -- (9,1.6) node[above]{$u$};
    \draw[->] (-0.5,-1.6) -- (-1.5,-1.6);
    \draw[->] (-2,-1.6) -- (-4,-1.6);
    \draw[->] (-5,-1.6) -- (-9,-1.6);
    \filldraw (3.6,-6) circle [radius=0.1];
    \filldraw (2.8,-6) circle [radius=0.1];
    \filldraw (2.0,-6) circle [radius=0.1];
    \end{tikzpicture}
    \caption{Outer/inner bumps of the De Gregorio blow up}
\end{subfigure}
\begin{subfigure}{0.5\textwidth}
    \centering
    \begin{tikzpicture}
    \draw[->] (-3,0) -- (3,0) node[below]{$r$} coordinate (x axis);
    \draw[->] (0,-3) -- (0,3) node[right]{$z$} coordinate (y axis);
    
    \filldraw (1.75,2) circle [radius=0.075];
    \filldraw (2.25,2) circle [radius=0.075];
    \filldraw (1.75,1.5) circle [radius=0.075];
    \filldraw (2.25,1.5) circle [radius=0.075];
    \filldraw (1.75,0.5) circle [radius=0.075];
    \filldraw (2.25,0.5) circle [radius=0.075];
    \filldraw (0.3,0.3) circle [radius=0.05];

    \draw[->] (0.2,0.7) to[bend right] (0.7,0.2);
    \draw[->] (0.2,1.4) to[bend right] (1.4,0.2) node[above]{$u$};

    \filldraw (-1.75,-2) circle [radius=0.075];
    \filldraw (-2.25,-2) circle [radius=0.075];
    \filldraw (-1.75,-1.5) circle [radius=0.075];
    \filldraw (-2.25,-1.5) circle [radius=0.075];
    \filldraw (-1.75,-0.5) circle [radius=0.075];
    \filldraw (-2.25,-0.5) circle [radius=0.075];
    \filldraw (-0.3,-0.3) circle [radius=0.05];
    
    \draw[->] (-0.2,-0.7) to[bend right] (-0.7,-0.2);
    \draw[->] (-0.2,-1.4) to[bend right] (-1.4,-0.2);

    \draw (1.675,-1.925) -- (1.825,-2.075);
    \draw (1.675,-2.075) -- (1.825,-1.925);
    \draw (2.175,-1.925) -- (2.325,-2.075);
    \draw (2.175,-2.075) -- (2.325,-1.925);
    \draw (1.675,-1.425) -- (1.825,-1.575);
    \draw (1.675,-1.575) -- (1.825,-1.425);
    \draw (2.175,-1.425) -- (2.325,-1.575);
    \draw (2.175,-1.575) -- (2.325,-1.425);
    \draw (1.675,-0.425) -- (1.825,-0.575);
    \draw (1.675,-0.575) -- (1.825,-0.425);
    \draw (2.175,-0.425) -- (2.325,-0.575);
    \draw (2.175,-0.575) -- (2.325,-0.425);
    \draw (0.25,-0.25) -- (0.35,-0.35);
    \draw (0.25,-0.35) -- (0.35,-0.25);
    
    \draw[->] (0.2,-0.7) to[bend left] (0.7,-0.2);
    \draw[->] (0.2,-1.4) to[bend left] (1.4,-0.2);

    \draw (-1.675,1.925) -- (-1.825,2.075);
    \draw (-1.675,2.075) -- (-1.825,1.925);
    \draw (-2.175,1.925) -- (-2.325,2.075);
    \draw (-2.175,2.075) -- (-2.325,1.925);
    \draw (-1.675,1.425) -- (-1.825,1.575);
    \draw (-1.675,1.575) -- (-1.825,1.425);
    \draw (-2.175,1.425) -- (-2.325,1.575);
    \draw (-2.175,1.575) -- (-2.325,1.425);
    \draw (-1.675,0.425) -- (-1.825,0.575);
    \draw (-1.675,0.575) -- (-1.825,0.425);
    \draw (-2.175,0.425) -- (-2.325,0.575);
    \draw (-2.175,0.575) -- (-2.325,0.425);
    \draw (-0.25,0.25) -- (-0.35,0.35);
    \draw (-0.25,0.35) -- (-0.35,0.25);
    
    \draw[->] (-0.2,0.7) to[bend left] (-0.7,0.2);
    \draw[->] (-0.2,1.4) to[bend left] (-1.4,0.2);

    \end{tikzpicture}
    \caption{Outermost region of the 3D Euler blow up\\
    $\bullet = \omega$ pointing outward, $\times = \omega$ pointing inward\\
    For $r < 0$, the coordinate $(r, \theta) = (|r|, \theta + \pi)$}
\end{subfigure}
\end{figure}

\subsection{Strategy of the proofs of Theorem 1 and 2}
Here we give a sketch of the blow-up mechanism behind Theorem 1 and how to extended to 3D Euler.

\subsubsection{Blow-up of the De Gregorio model}
We take the ansatz
\[
w(x, t) = \sum_{k=0}^\infty w_k(x, t).
\]
Then for the De Gregorio equation to hold, we only need to arrange that
\[
\partial_tw_k + u\partial_xw_k = w_kHw.
\]
We let $w_k$ re-scale as follows:
\[
w_k(x, t) = x_k(t)W_k(c_k(t)x, t).
\]
Assume that in the whole evolution, $W_k$ consists of two bumps supported near $\pm1$,
with an odd symmetry with respect to the origin.
Then $w_k$ is supported near $\pm1/c_k(t)$, whose size is proportional to $x_k(t)$.
Also assume that $c_{k+1}(t) \gg c_k(t)$.
Then as $k$ gets larger, supp $w_k$ gets closer to the origin.

We decompose $w = w_k + w_- + w_+$, where
\[
w_- = \sum_{j<k} w_j, \quad w_+ = \sum_{j>k} w_j.
\]
Similar decomposition can be made of $Hw$ and $u$.
Conceptually $w_-$ captures the contribution of all the bumps lying further away from the origin than $w_k$,
and $w_+$ captures the contribution of those lying closer to the origin than $w_k$.

We now study how they affect the evolution of $w_k$.
The latter contribution from the inner bumps is insignificant,
because it is from an odd function supported very close to the origin
(compared to $w_k$ itself), whose influences mostly cancel each other.

The effect of the outer bumps can be captured in the equation
\[
\partial_tw_k + u_-\partial_xw_k = w_kHw_-, \quad \partial_xu_- = Hw_-.
\]
Since $w_k$ is supported much closer to the origin than $w_-$,
up to a small error,
\[
\partial_tw_k + xHw_-(0, t)\partial_xw_k = Hw_-(0, t)w_k.
\]
We can integrate this equation using the method of characteristics to get
\[
w_k(x, t) = x_k(t)w_k(x/x_k(t), 0)
\]
where $x_k(t) > 0$ tracks the height of the bump $w_k$ and satisfies the equation
\[
\dot x_k(t) = x_k(t)Hw_-(0, t) = x_k(t)\sum_{j=0}^{k-1} Hw_j(0, t)
= x_k(t)\sum_{j=0}^{k-1} x_j(t)HW_j(0, t).
\]
We arrange that $HW_j(0, t)$ is close to 1 to get the following system of ODEs:
\[
\dot x_k(t) = x_k(t)(x_0(t) + \dots + x_{k-1}(t)),
\]
which is a closed, self-contained model of the heights of the bumps $w_k$.

\subsubsection{The ODE model and its explicit solution}
We now solve the ODE model explicitly by integration.
Since $x_k(t) > 0$, we can let $y_k(t) = \ln x_k(t)$. Then
\[
\dot y_k(t) = e^{y_0(t)} + \dots + e^{y_{k-1}(t)} = \dot y_{k-1}(t) + e^{y_{k-1}(t)}.
\]
Then
\[
y_k(t) = d_k + y_{k-1}(t) + \int_0^t e^{y_{k-1}(s)}ds.
\tag{$d_k = y_k(0) - y_{k-1}(0)$}
\]
Let
\[
z_k(t) = \int_0^t x_k(s)ds = \int_0^t e^{y_k(s)}ds.
\]
Then
\[
z_k = e^{d_k}\int_0^t e^{y_{k-1}(s)}e^{\int_0^s e^{y_{k-1}(r)}dr}ds
= e^{d_k}\int_0^t \dot z_{k-1}(s)e^{z_{k-1}(s)}ds
= e^{d_k}(e^{z_{k-1}(t)} - 1).
\]
Let $x_k(0) = A^k$ where $A > 1$, so that $d_k = \ln A > 0$. Then
\[
z_k(t) = A(e^{z_{k-1}(t)} - 1).
\]
From this and the recurrence relation we can recover $x_k(t)$ explicitly.
Nevertheless, it is more important to study its asymptotics as $k \to \infty$, which we now do.

Let $f(x) = A(e^x - 1)$. Then $f(0) = 0$. If $x > 0$ then $f(x) > Ax > x$.
From $f'(x) = Ae^x$ it is easy to see that there is a unique $-a < 0$ such that $f(-a) = -a$,
so for $t < 0$, $z_k(t) \to -a$ as $k \to \infty$. Then for fixed $t < 0$,
\[
y_k(t) - y_k(0) = \sum_{j=0}^{k-1} z_k(t) = -(a + o_t(1))k
\]
so $y_k(t) = (\ln A - a + o_t(1))k$. Solving the equation
\[
-a = f(-a) = A(e^{-a} - 1) = A(-a + a^2/2 + \dots)
\]
we get $a = 2(A - 1) + \dots$, so $y_k(t) = (-(A - 1) + \dots)k$.
The support of $w_k(\cdot, t)$ is roughly that of $w_k(\cdot, 0)$,
scaled by a factor of $x_k(t)/x_k(0)$. Assume the $w_k(\cdot, 0)$
is supported near $\pm r^k$, where $r \ll 1$. Then $w_k(\cdot, t)$
is supported near $\pm r^kx_k(t)/x_k(0)$, so for $t < 0$,
the H\"older exponent of $w(\cdot, t)$ at 0 is
\[
\frac{\ln x_k(t)}{\ln r^kx_k(t)/x_k(0)} = \frac{y_k(t)}{k\ln r+y_k(t)-y_k(0)}
= \frac{-((A - 1) + \cdots)k}{(\ln r - 2(A - 1) + \cdots)k} > 0.
\]
Meanwhile at time 0 $w(\cdot, 0)$ is unbounded because $x_k(0) = A^k \to \infty$ as $k \to \infty$.

We still need to account for the self-interaction of $w_k$, which can be modeled by the equation
\[
\partial_tw_k + u_k\partial_xw_k = w_kHw_k, \quad \partial_xu_k = Hw_k.
\]
Schematically it is of the form $\partial_tw_k = O(w_k^2)$ and has an energy estimate like
\[
\frac{dE(t)}{dt} = x_k(t)E(t)
\]
because the height of $Hw_k$ is proportional to $x_k$. This integrates to
\[
E(t) \le e^{\int_t^0 x_k(s)ds}E(0).
\]
Note that the integral is exactly $-z_k(t) \to a \sim 2(A - 1)$,
indicating that the effect of self-interaction is negligible if $A$ is very close to 1.
Indeed, in this case, the heights of the bumps grow very slowly as one approaches the origin,
so the effects of the outer bumps sum up like a geometric sequence
with ratio $1/A$, so their combined damping effect dominates that of self-interaction.
For the same reason, the H\"older exponent, which is proportional to $A - 1$
according to the computation above, is also very small.

\subsubsection{The limiting argument}
Since there is no \textit{a priori} wellposedness result for the unbounded data
\[
w(x, 0) = \sum_{k=0}^\infty A^kW_k(x/r^k, 0),
\]
we work instead with the data
\[
w_n(x, 0) = \sum_{k=0}^n A^kW_{n,k}(x/r^k, 0),
\]
which is a finite sum and hence smooth if we make $W_{n,k}(\cdot, 0)$ smooth.
Then local wellposedness of smooth solutions is trivial.
We use energy estimates to show a life span independent of $n$,
in which the uniform boundedness of the time integral of $x_n$ plays a key role.
This allows us to obtain a uniform bound of the H\"older norms of $w_n(\cdot, t)$
for $t < 0$ and pass to the limit $n \to \infty$ to show a blow up in the sup norm
from H\"older continuous initial data.

\subsubsection{Blow-up of the 3D axi-symmetric Euler without swirl}
We work with the vorticity formulation of the Euler equation.
In the swirless case, the vorticity only has the angular component $w(r, z, t)e_\theta$,
which satisfies
\[
\partial_tw + u^r\partial_rw + u^z\partial_zw = u^rw/r
\]
which is very similar to the De Gregorio model: the left-hand side contains the multi-variable analog
of the advection term $u\partial_xw$, where the gradient of $u$ is also related to $w$ via a singular integral operator;
the right-hand side is approximately the vortex stretching term $w\partial_ru^r$ after making the approximation
$u^r/r = (u^r - u^r(0))/r \approx \partial_ru^r$. Thus we make the ansatz
\[
w(r, z, t) = \sum_{k=0}^\infty w_k(r, z, t)
\]
where the support of $w_{k+1}$ is much closer to the origin than that of $w_k$.
Then $w_{k+1}$ is mostly affected by the velocity field generated at the origin by $w_0, \dots, w_k$.
For axi-symmetric vorticity $w_k$, the $z$ component of the velocity along the $z$ axis is
\[
u^z(0, z) = -\int_{-\infty}^\infty \int_0^\infty \frac{r^2w_k(r, z')}{2((z - z')^2 + r^2)^{3/2}}drdz'.
\]
Differentiation with respect to $z$ and evaluation at 0 shows that
\[
\partial_zu^z(0, 0) = -\frac{3}{2}\int_{-\infty}^\infty \int_0^\infty \frac{r^2zw_k(r, z)}{(z^2 + r^2)^{5/2}}drdz.
\]
By symmetry and incompressibility,
\[
\partial_ru^r(0, 0) = -\frac12\partial_zu^z(0, 0)
= \frac{3}{4}\int_{-\infty}^\infty \int_0^\infty \frac{r^2zw_k(r, z)}{(z^2 + r^2)^{5/2}}drdz.
\]
We assume that $w$ is odd with respect to $z$ and is positive when $z > 0$.
Then $\partial_ru^r(0, 0) > 0$ and $\partial_zu^z(0, 0) = -2\partial_ru^r(0, 0) < 0$.
% The flow around the origin is a saddle with compression in the $z$ direction
% and expansion in the $r$ direction. The latter leads to the growth of $w$ because $w/r$ is conserved along the flow.
Then each individual piece can be assumed to take the form
\[
w_k(r, z, t) = x_k(t)W_k\left( \frac{A^kr}{x_k(t)d^k}, \frac{x_k(t)^2z}{A^{2k}d^k}, t \right)
\]
where $\dot x_k/x_k = \partial_ru^r(0, 0)$ and $W_k$ does not change much in time.
Since $x_k$ increases in $t$, a blow up can be set up at $t = 0$ from $t < 0$.

As in the De Gregorio case we need to estimate $\dot x_k/x_k$ to get an ODE system,
which amounts to estimating the integral in the expression of $\partial_ru^r(0, 0)$.
We will set up the initial
data so that $W_k$ is supported near $r = 1$ and vanishes when $z = 0$.
Then $w_k$ is supported near $r = x_k(t)(d/A)^k$ and vanishes when $z = 0$.
Since the kernel is homogeneous in $r$ and $z$ and decays like $1/z^4$ as $z \to \infty$,
the integral is dominated by the part where $r \le Cz$. Since $w_k$ vanishes when $z = 0$,
a typical sample point is $r = z = x_k(t)(d/A)^k$, where
\[
w_k(r, z, t)
= w_k\left( \frac{x_k(t)d^k}{A^k}, \frac{x_k(t)d^k}{A^k}, t\right)
= x_k(t)W_k\left( 1, \frac{x_k(t)^3}{A^{3k}}, t \right).
\]
If we set $|W(r, z, 0)|$ proportional to $|z|^\alpha$ (when $r$ is fixed),
then the above is proportional to $x_k(t)^{1+3\alpha}/A^{3\alpha k}$,
compared to $x_k(t)$ in the De Gregorio model. To minimize this deviation,
we will let $\alpha$ be small (1/12 to be precise) and
modify the estimates on the ODE model correspondingly
to allow for a small power of $x_k(t)/A^k$.
After that, the rest is quite standard:
we control the evolution of the piece in the Lagrangian coordinates
using bounds on the flow map and its differential;
see the bootstrap assumptions for details.

Now we examine the regularity of our solution.
Note that $\supp w_k(\cdot, \cdot, t)$ converges to the origin, i.e.,
any point except the origin has a neighborhood which intersects only the supports of finitely many $w_k$,
so the regularity of $w = \sum_k w_k$ is the same as each individual summand,
i.e., it is smooth everywhere but has H\"older type singularities on the plane $z = 0$.
We now smooth $W_k(r, z, 0)$ out near $z = 0$. Since the time integral of $x_k$ is finite,
the evolution of $w_k$ will not differ much if the modification is sufficiently localized.
This way we remove the singularity on the plane $z = 0$,
except the one at the origin.

\subsection{Outline of the paper}
The rest of the paper is organized as follows. In section 2 we prove Theorem 1 following the scheme described above and in section 3 we extend the strategy to higher dimension and prove Theorem 2.

\section{The De Gregorio equation}
In this part we first derive in subsection \ref{1DODE} estimates related to a more general version of the explicit ODE model. In subsection \ref{1Dansatz} we specify the exact form of the blow up to be constructed. Subsection \ref{1Dprofile} contains some useful estimates related to profiles of the blow up. Subsection \ref{1Denergy} is dedicated to energy estimates of the profiles, which allows us to construct the actual blow up using a limiting argument, as done in subsection \ref{1Dlimit}. Subsection \ref{1Dprop} explores various aspects of the blow up, including its H\"older continuity, the rate of the blow up, and its non-asymptotic similarity. Finally, in subsection \ref{1Dgen}, we extend the construction to the generalized De Gregorio equation.

\subsection{Estimates related to the ODE model}\label{1DODE}
In this subsection we study the ODE model in a more general setting,
focusing especially on the time integral of $x_k$, which will play a key role in our proof.
Lemma \ref{int-le} bounds this integral from above uniformly,
which enables us to control the self-interaction of $w_k$-
Lemma \ref{int-ge} bounds it from below,
and will be used to show the H\"older continuity of the solution.
\begin{lem}\label{int-le}
Assume that
\[
x_n(0) = A^n, \quad \dot x_n(t) = x_n(t)\sum_{j=0}^{n-1} a_j(t)x_j(t) \tag{$n = 0, 1, \cdots$}
\]
where $A > 1$ and $a_j(t) \ge 1$. Let $a > 0$ solve $a = A(1 - e^{-a})$.
Then $\ln A < a < 2(A - 1)$ and
\[
\int_{-a}^0 x_n(t)dt \le a. \tag{$n = 0, 1, \cdots$}
\]
\end{lem}
\begin{rem}
Note that the coefficients $a_j$ do not depend on $n$,
a feature that is important in the proof below.
\end{rem}
\begin{rem}
Rescaling in time makes this lemma applicable for other ranges of $a_{j}$.
For example, if $a_{j} \ge 2$, then one should replace $a$ with $a/2$
in the final estimate.
\end{rem}
\begin{proof}
First note that the equation for $a$ is equivalent to $a/(1 - e^{-a}) = A$.
The left-hand side tends to $1$ as $a \to 0+$,
and strictly increases to infinity as $a \to +\infty$,
so for any $A > 1$ the equation has a unique positive root.
Moreover, $A = a/(1 - e^{-a}) > 1 + a/2$, so $a < 2(A - 1)$.
On the other hand, since $\ln A < A - 1 = A(1 - e^{-\ln A})$,
$\ln A/(1 - e^{-\ln A}) < A$, so $a > \ln A$.

Now we bound the integral. Clearly $x_n > 0$ for all time. Let $x_n(t) = e^{y_n(t)}$. Then
\[
y_n(0) = n\ln A, \quad \dot y_n(t) = \sum_{j=0}^{n-1} a_j(t)e^{y_j(t)}. \tag{$n = 0, 1, \cdots$}
\]
Since $a_{j}(t) \ge 1$,
\[
\dot y_{n+1}(t) \ge \dot y_n(t) + e^{y_n(t)}.
\]
Integrating from $t \le 0$ to 0 we get
\[
y_{n+1}(t) \le \ln A + y_n(t) - \int_t^0 e^{y_n(s)}ds
\]
so
\[
e^{y_{n+1}(t)} \le Ae^{y_n(t)}e^{-\int_t^0 e^{y_n(s)}ds}.
\]
Integrating in time once again we get
\[
\int_{-a}^0 e^{y_{n+1}(t)}dt \le A(1 - e^{-\int_{-a}^0 e^{y_n(t)}dt}).
\]
Chaining this bound in $n$ we get
\[
\int_{-a}^0 x_n(t)dt = \int_{-a}^0 e^{y_n(t)}dt \le Z_n
\]
where the sequence $Z_n$ satisfies
\[
Z_0 = \int_{-a}^0 e^{y_0(t)}dt = a
\]
because $y_0(t) = y_0(0) = 0$, and $Z_{n+1} = A(1 - e^{-Z_n})$. Then $Z_n = a$ for all $n$.
\end{proof}

\begin{lem}\label{int-ge}
Assume that
\[
x_n(0) = A^n, \quad \dot x_n(t) = x_n(t)\sum_{j=0}^{n-1} a_j(t)x_j(t) \tag{$n = 0, 1, \cdots$}
\]
where $A > 1$ and $a_{j}(t) \in [0, 1]$. Then for all $t \le 0$,
\[
\int_t^0 x_n(s)ds \ge I_n(t) \tag{$n = 0, 1, \cdots$}
\]
where $I_0(t) = t$ and $I_{n+1}(t) = A(1 - e^{-I_n(t)})$.
\end{lem}
\begin{rem}
Let $a$ be as in Lemma \ref{int-le}. If $I_0(t) < a$ (resp. $> a$),
then $I_n(t)$ increases (resp. decreases) to $a$ as $n \to \infty$.
\end{rem}

\begin{proof}
Like the proof above, we have
\[
\dot y_{n+1}(t) \le \dot y_n(t) + e^{y_n(t)}.
\]
Integrating from $t \le 0$ to 0 we get
\[
y_{n+1}(t) \ge \ln A + y_n(t) - \int_t^0 e^{y_n(s)}ds
\]
so
\[
e^{y_{n+1}(t)} \ge Ae^{y_n(t)}e^{-\int_t^0 e^{y_n(s)}ds}.
\]
Integrating in time once again we get
\[
\int_t^0 e^{y_{n+1}(s)}ds \ge A(1 - e^{-\int_t^0 e^{y_n(s)}ds}).
\]
Chaining this bound in $n$ we get
\[
\int_t^0 x_n(s)ds \ge I_n(t)
\]
where the sequence $I_n(t)$ satisfies
\[
I_0(t) = \int_t^0 e^{y_0(s)}ds = t
\]
because $y_0(t) = y_0(0) = 0$, and $I_{n+1}(t) = A(1 - e^{-I_n(t)})$.
\end{proof}

\begin{lem}\label{monotone}
%If in the ODE system above $a_{n,j}(t)$ is independent of $n$, then f
For $n > k$,
the ratio $x_n(t)/x_k(t)$ is increasing in $t$.
\end{lem}
\begin{proof}
Note that $\partial_t\ln(x_n(t)/x_k(t)) = \dot y_n(t) - \dot y_k(t) = \sum_{j=k}^{n-1} a_j(t)x_j(t) > 0$.
\end{proof}

\subsection{The ansatz}\label{1Dansatz}
Now we spell out the explicit form of the blow up.
Let $\rho$ be a non-negative bump function supported in $[1 - r, 1 + r]$
($r \le 1/4$) with $H\rho(0) = -1/2$, and symmetric with respect to 1.
Let $\phi = \rho(-x) - \rho(x)$. Then $H\phi(0) = 1$. Consider the data
\[
w_n(x, 0) = \sum_{k=0}^n w_{n,k}(x, 0),\quad
w_{n,k}(x, 0) = A^k\phi(x/r^k),
\]
where $A > 1$ is to be determined later.
The evolution equation
\[
w_t + uw_x = wHw,\quad u_x = Hw
\]
is satisfied if
\[
w_n(x, t) = \sum_{k=0}^n w_{n,k}(x, t)
\]
and
\[
\partial_tw_{n,k} + u_n\partial_xw_{n,k} = w_{n,k}Hw_n,\quad
\partial_xu_n = Hw_n.
\]

Let
\[
w_{n,k}(x, t) = x_{n,k}(t)W_{n,k}\left( \frac{A^kx}{x_{n,k}(t)r^k}, t \right)
\]
where $x_{n,k}(t)$ are to be determined later. Then
\[
W_{n,k}(x, t) = \frac{w_{n,k}(x(r/A)^kx_{n,k}(t), t)}{x_{n,k}(t)}.
\]
We take the initial data
\[
x_{n,k}(0) = A^k,\quad W_{n,k}(x, 0) = A^{-k}w_{n,k}(xr^k, 0) = \phi(x).
\]
Then
\begin{align*}
\partial_tw_{n,k}(x, t)
&= \dot x_{n,k}(t)W_{n,k}((x/x_{n,k}(t))(A/r)^k, t)\\
&- (x\dot x_{n,k}(t)/x_{n,k}(t))(A/r)^k\partial_xW_{n,k}((x/x_{n,k}(t))(A/r)^k, t)\\
&+ x_{n,k}(t)\partial_tW_{n,k}((x/x_{n,k}(t))(A/r)^k, t)\\
&= (\dot x_{n,k}(t)/x_{n,k}(t))w_{n,k}(x, t)\\
&- (x\dot x_{n,k}(t)/x_{n,k}(t))\partial_xw_{n,k}(x, t)\\
&+ x_{n,k}(t)\partial_tW_{n,k}((x/x_{n,k}(t))(A/r)^k, t).
\end{align*}
On the other hand,
\[
\partial_tw_{n,k} + u_n\partial_xw_{n,k} = w_{n,k}Hw_n,\quad
\partial_xu_n = Hw_n
\]
so
\begin{align*}
x_{n,k}(t)\partial_tW_{n,k}((x/x_{n,k}(t))(A/r)^k, t)
&= (Hw_n(x, t) - \dot x_{n,k}(t)/x_{n,k}(t))w_{n,k}(x, t)\\
&- (u_n(x, t) - x\dot x_{n,k}(t)/x_{n,k}(t))\partial_xw_{n,k}(x, t).
0\end{align*}
We decompose $w_n = w_{n,k} + w_-+ w_+$, where
\[
w_- = \sum_{j<k} w_{n,j}, \quad w_+ = \sum_{j>k} w_{n,j}.
\]
We let
\[
\dot x_{n,k}(t) = x_{n,k}(t)Hw_-(0, t)
= x_{n,k}(t)\sum_{j=0}^{k-1} x_{n,j}(t)HW_{n,j}(0, t).
\]
This system of ODE is wellposed as long as the solution still exists, because the first $k + 1$ equations form a closed system in terms of $x_{n,0}, \dots, x_{n,k}$ and the new unknown $x_{n,k}$ appears linearly in the last equation.

\subsection{The profiles}\label{1Dprofile}
We use a bootstrap argument to show the wellposedness of the evolution of the profiles $W_{n,k}$.
The bootstrap assumptions are: for $t \in [-T, 0]$ and $k = 0, \dots, n$,
\begin{enumerate}
    \item $W_{n,k}(\cdot, t) = 0$ outside $[-1 - 2r, -1 + 2r] \cup [1 - 2r, 1 + 2r]$.
    \item $\|\partial_x(W_{n,k}(\cdot, t) - \phi)\|_{L^2} \le \epsilon$, where $\epsilon > 0$ is to be determined later. 
\end{enumerate}

\subsubsection{Estimates related to the profiles}
From the bootstrap assumptions we can obtain estimates of the profiles and their Hilbert transforms.
\begin{lem}\label{W-phi-L1}
$\|W_{n,k}(\cdot, t) - \phi\|_{L^1} \le \epsilon\sqrt{32r^3/3}$.
\end{lem}
\begin{proof}
Applying the fundamental theorem of calculus to $f = W_{n,k}(\cdot, t) - \phi$
and using the Cauchy's inequality we get
\[
\|f\|_{L^1} \le \|1 - |x|\|_{L^2(\supp f)}\|f'\|_{L^2}
\le \sqrt{4\int_0^{2r} x^2dx}\|f'\|_{L^2} = \sqrt{32r^3/3}\|f'\|_{L^2}.
\]
\end{proof}
Since $H\phi(0) = 1$, from Lemma \ref{W-phi-L1} it follows that
\[
|HW_{n,k}(0, t) - 1| \le \frac{\sqrt{32r^3/3}}{\pi(1 - 2r)}\epsilon =: c(r)\epsilon
\]
so Lemma \ref{int-le} and Lemma \ref{int-ge} apply with
%$b = (1 + c(r)\epsilon)/(1 - c(r)\epsilon)$ (resp. its reciprocal) and 
a suitable rescaling of time.

Next we estimate
\begin{align*}
Hw_n(x, t) - \dot x_{n,k}(t)/x_{n,k}(t)
&= Hw_{n,k}(x, t) + Hw_-(x, t) - \dot x_{n,k}(t)/x_{n,k}(t)\\
&+ Hw_+(x, t).
\end{align*}
\begin{lem}\label{Hwpm-Loo}
There are constants $C_i(r, \epsilon)$ (i = 1, 2) such that for $A \in (1, 2)$,
and $x \in \supp w_{n,k}(\cdot, t)$, namely $(|x|/x_{n,k}(t))(A/r)^k \in [1 - 2r, 1 + 2r]$,
\begin{align*}
|\partial_xHw_-(x, t)| &\le C_1(r, \epsilon)(A/r)^k,\\
|Hw_-(x, t) - \dot x_{n,k}(t)/x_{n,k}(t)| &\le (1 + 2r)C_1(r, \epsilon)x_{n,k}(t),\\
|v_-(x, t)| &\le (1 + 2r)^2C_1(r, \epsilon)x_{n,k}(t)^2(r/A)^k, \tag{$v_-(0) = 0$, $\partial_xv_-=Hw_-(x, t) - \dot x_{n,k}(t)/x_{n,k}(t)$}\\
|v_+(x, t)| &\le 7(1 - 2r)C_2(r, \epsilon)x_{n,k}(t)^2(r/A)^k/4, \tag{$v_+(0) = 0$, $\partial_xv_+=Hw_+$}\\
|Hw_+(x, t)| &\le C_2(r, \epsilon)x_{n,k}(t),\\
|\partial_xHw_+(x, t)| &\le 32(A/r)^kC_2(r, \epsilon)/(7 - 14r).
\end{align*}
\end{lem}
\begin{proof}
Using $\sgn \phi(x) = \sgn x$ and the symmetry of $\rho$ with respect to 1,
we have that $\|\phi\|_{L^1} \le H\phi(0) = 1$. Then by Lemma \ref{W-phi-L1},
\[
\|W_{n,k}(\cdot, t)\|_{L^1} \le 1 + \pi(1 - 2r)c(r)\epsilon
\]
so for $x \in [0, 1 - 2r]$,
\[
|\partial_xHW_{n,k}(\pm x, t)| \le
\frac{1 + \pi(1 - 2r)c(r)\epsilon}{\pi(1 - 2r - x)^2}.
\]

For $j < k$ and $(|x|/x_{n,k}(t))(A/r)^k \in [1 - 2r, 1 + 2r]$ we have that
\begin{align*}
|\partial_xHw_{n,j}(x, t)|
&= (A/r)^j\left| \partial_xHW_{n,j}\left( \frac{A^jx}{x_{n,j}(t)r^j}, t \right) \right|\\
&\le \frac{(1 + \pi(1 - 2r)c(r)\epsilon)(A/r)^j}
{\pi\left( 1 - 2r - \frac{(1 + 2r)x_{n,k}(t)r^{k-j}}{x_{n,j}(t)A^{k-j}} \right)^2}.
\end{align*}
By Lemma \ref{monotone}, for $t \le 0$, $x_{n,k}(t)/x_{n,j}(t) \le x_{n,k}(0)/x_{n,j}(0) = A^{k-j}$, so\begin{align*}
|\partial_xHw_{n,j}(x, t)|
&\le \frac{(1 + \pi(1 - 2r)c(r)\epsilon)(A/r)^j}
{\pi(1 - 2r - (1 + 2r)r^{k-j})^2}\\
&\le \frac{(1 + \pi(1 - 2r)c(r)\epsilon)(A/r)^j}
{\pi(1 - 3r - 2r^2)^2}.
\end{align*}
Note that since $r \le 1/4$, $1 - 3r - 2r^2 > 0$,
so the bound on $\partial_xW$ is indeed applicable. 
Summing over $j = 0, \dots, k - 1$ we have that
\begin{align*}
|\partial_xHw_-(x, t)|
&\le \frac{(1 + \pi(1 - 2r)c(r)\epsilon)(r/A)}
{\pi(1 - 3r - 2r^2)^2(1 - r/A)}(A/r)^k\\
&:= C_1(r, \epsilon)(A/r)^k.
\end{align*}
Note that $C_1$ is uniform in $A \in (1, 2)$.
The bounds on $Hw_-(x, t) - \dot x_{n,k}(t)/x_{n,k}(t)$ and $v_-$
follows form integrating from 0 to $|x| \le x_{n,k}(t)(r/A)^k$.

For $j > k$, using the $L^1$ norm of $W$ above we have that
for $|x| > 1 + 2r$,
\begin{align*}
\left| \int_0^x HW_{n,j}(y, t)dy \right|
&\le \frac{1 + \pi(1 - 2r)c(r)\epsilon}{2\pi}\ln\frac{|x| + 1 + 2r}{|x| - 1 - 2r},\\
&\le \frac{(1 + \pi(1 - 2r)c(r)\epsilon)(1 + 2r)}{\pi(|x| - 1 - 2r)}\\
|HW_{n,j}(x, t)| &\le \frac{1 + \pi(1 - 2r)c(r)\epsilon}{2\pi}
\left( \frac{1}{|x| - 1 - 2r} - \frac{1}{|x| + 1 + 2r} \right)\\
&= \frac{(1 + \pi(1 - 2r)c(r)\epsilon)(1 + 2r)}{\pi(x^2 - (1 + 2r)^2)},\\
|\partial_xHW_{n,j}(x, t)| &\le \frac{1 + \pi(1 - 2r)c(r)\epsilon}{2\pi}
\left( \frac{1}{(|x| - 1 - 2r)^2} - \frac{1}{(|x| + 1 + 2r)^2} \right)\\
&= \frac{2(1 + \pi(1 - 2r)c(r)\epsilon)(1 + 2r)|x|}{\pi(x^2 - (1 + 2r)^2)^2}
\end{align*}
so
\begin{align*}
|Hw_{n,j}(x, t)|
&= x_{n,j}(t)\left| HW_{n,j}\left( \frac{A^jx}{x_{n,j}(t)r^j}, t \right) \right|\\
&\le \frac{(1 + \pi(1 - 2r)c(r)\epsilon)(1 + 2r)}
{\pi\left( \left( \frac{A^j|x|}{x_{n,j}(t)r^j} \right)^2 - (1 + 2r)^2 \right)}x_{n,j}(t).
\end{align*}
For $(|x|/x_{n,k}(t))(A/r)^k \in [1 - 2r, 1 + 2r]$ we have that, again by Lemma \ref{monotone},
\[
\frac{A^j|x|}{x_{n,j}(t)r^j} \ge \frac{(1 - 2r)A^{j-k}x_{n,k}(t)}{x_{n,j}(t)r^{j-k}}
\ge (1 - 2r)r^{k-j}.
\]
Since
\[
\frac{(1 - 2r)r^{k-j}}{1 + 2r} \ge \frac{1 - 2r}{r(1 + 2r)} \ge \frac43,
\]
it follows that
\[
\left( \frac{A^j|x|}{x_{n,j}(t)r^j} \right)^2 - (1 + 2r)^2
\ge \frac{7}{16}\left( \frac{A^j|x|}{x_{n,j}(t)r^j} \right)^2
\ge \frac{7}{16}(1 - 2r)^2r^{2(k-j)}
\]
so
\[
|Hw_{n,j}(x, t)| \le \frac{16(1 + \pi(1 - 2r)c(r)\epsilon)(1 + 2r)(Ar^2)^{j-k}}
{7\pi(1 - 2r)^2}x_{n,k}(t)
\]
and similarly,
\[
\left| \int_0^x Hw_{n,j}(y, t)dy \right| \le \frac{4(1 + \pi(1 - 2r)c(r)\epsilon)(1 + 2r)(Ar^2)^{j-k}}
{\pi(1 - 2r)}x_{n,k}(t)^2(r/A)^k
\]
and
\begin{align*}
|\partial_xHw_{n,j}(x, t)|
&= (A/r)^j\left| \partial_xHW_{n,j}\left( \frac{A^jx}{x_{n,j}(t)r^j}, t \right) \right|\\
&\le 2(1 + \pi(1 - 2r)c(r)\epsilon)(1 + 2r)(A/r)^j\\
&\times \frac{4}{\sqrt7\pi}\left( \left( \frac{A^j|x|}{x_{n,j}(t)r^j} \right)^2 - (1 + 2r)^2 \right)^{-3/2}\\
&\le \frac{2^9(1 + \pi(1 - 2r)c(r)\epsilon)(1 + 2r)(r^3)^{j-k}(A/r)^j}
{49\pi(1 - 2r)^3}.
\end{align*}
Applicability of the bound on $W$ is again justified by the positivity of the denominator.
Summing over $j > k$ we have that
\begin{align*}
|Hw_+(x, t)|
&\le \frac{16(1 + \pi(1 - 2r)c(r)\epsilon)(1 + 2r)(Ar^2)}
{7\pi(1 - 2r)^2(1 - Ar^2)}x_{n,k}(t)\\
&=: C_2(r, \epsilon)x_{n,k}(t),\\
|v_+(x. t)| &\le 7(1 - 2r)C_2(r, \epsilon)x_{n,k}(t)^2(r/A)^k/4,\\
|\partial_xHw_+(x, t)|
&\le \frac{32(A/r)^k}{7(1 - 2r)}C_2(r, \epsilon).
\end{align*}
Similarly $C_2$ is uniform in $A \in (1, 2)$.
\end{proof}

\subsubsection{Evolution of the profiles}
The evolution equation of $W_{n,k}$ is
\[
\partial_tW_{n,k}(x, t) = v_{n,k}(x, t)W_{n,k}(x, t) - V_{n,k}(x, t)\partial_xW_{n,k}(x, t)
\]
where
\begin{align*}
v_{n,k}(x, t) &= x_{n,k}(t)HW_{n,k}(x, t) + Hw_-(xx_{n,k}(t)(r/A)^k, t) - \dot x_{n,k}(t)/x_{n,k}(t)\\
&+ Hw_+(xx_{n,k}(t)(r/A)^k, t)
\end{align*}
and $\partial_xV_{n,k} = v_{n,k}$ with $V_{n,k}(0, t) = 0$. Then
\[
\partial_{tx}W_{n,k}(x, t) = \partial_xv(x, t)W_{n,k}(x, t) - V(x, t)\partial_x^2W_{n,k}(x, t).
\]

We now estimate $v_{n,k}$.
\begin{lem}\label{W-phi-sup}
$\sup|HW_{n,k}(\cdot, t) - H\phi| \le 2\epsilon\sqrt{r/\pi}$.
\end{lem}
\begin{proof}
By \cite[(1.20)]{Liu-Wang}, $\sup|f| \le \sqrt{\|f\|_{L^2}\|f'\|_{L^2}}$.
% If $\|f\|_{L^2} = \|xf\|_{L^2} = 1$, then
% \[
% \int f \le \sqrt{\int (1 + x^2)f(x)^2dx}\sqrt{\int \frac{dx}{1 + x^2}} = \sqrt{2\pi}.
% \]
% By rescaling we have that $\int f \le \sqrt{2\pi\|f\|_{L^2}\|xf\|_{L^2}}$.
% Taking the Fourier transform we get that
% $\sup|f| \le \int \hat f \le \sqrt{2\pi\|\hat f\|_{L^2}\|\xi\hat f\|_{L^2}}
% \le \sqrt{\|f\|_{L^2}\|f'\|_{L^2}}$.
Apply this to $f=HW_{n,k}(\cdot, t) - H\phi$, using the unitarity of $H$,
$\|f'\|_{L^2} \le \epsilon$ and $\|f\|_{L^2} \le 4r\epsilon/\pi$ from Poincar\'e's inequality
and $|\supp f| \le 8r$.
\end{proof}

\begin{lem}\label{v-Loo}
There are constants $C_i(r, \epsilon)$ (i = 3, 4, 5) such that for $A \in (1, 2)$,
on $\supp W_{n,k}(\cdot, t) \sub [1 - 2r, 1 + 2r]$,
\begin{align*}
|v_{n,k}(x, t)| &\le x_{n,k}(t)(\max|H\phi| + C_3(r, \epsilon)),\\
|V_{n,k}(x, t)| &\le x_{n,k}(t)(\max|\Phi| + C_4(r, \epsilon)), \tag{$\Phi(0) = 0$, $\Phi_x = H\phi$}\\
\|\partial_xv_{n,k}(\cdot, t)\|_{L^2} &\le x_{n,k}(t)(\|\phi'\|_{L^2} + C_5(r, \epsilon)).
\end{align*}
\end{lem}
\begin{proof}
By Lemma \ref{W-phi-sup}, $|HW_{n,k}(x, t)| \le \max|H\phi| + 2\epsilon\sqrt{r/\pi}$.
Then by Lemma \ref{Hwpm-Loo},
\begin{align*}
|v_{n,k}(x, t)| &\le x_{n,k}(t)(\max|H\phi| + C_3(r, \epsilon)),\\
|V_{n,k}(x, t)| &\le x_{n,k}(t)(\max|\Phi| + C_4(r, \epsilon))
\end{align*}
where
\begin{align*}
C_3(r, \epsilon) &= 2\epsilon\sqrt{r/\pi} + (1 + 2r)C_1(r, \epsilon) + C_2(r, \epsilon),\\
C_4(r, \epsilon) &= 2\epsilon(1 + 2r)\sqrt{r/\pi} + (1 + 2r)^2C_1(r, \epsilon) + 7C_2(r, \epsilon)/4,
\end{align*}

For the $L^2$ bound, from
\[
\frac{\partial_xv_{n,k}(x, t)}{x_{n,k}(t)}
= \partial_xHW_{n,k}(x, t) + (r/A)^k\sum \partial_xHw_\pm(xx_{n,k}(t)(r/A)^k, t),
\]
and $\|\partial_xHW_{n,k}(\cdot, t)\|_{L^2} = \|\partial_xW_{n,k}(\cdot, t)\|_{L^2} \le \|\phi'\|_{L^2} + \epsilon$
it follows that
\[
\frac{\|\partial_xv_{n,k}(\cdot, t)\|_{L^2}}{x_{n,k}(t)}
\le \|\phi'\|_{L^2} + \epsilon + (r/A)^k\sum \|\partial_xHw_\pm(xx_{n,k}(t)(r/A)^k, t)\|_{L^2}.
\]
Since only $\supp W_{n,k}(\cdot, t)$ is relevant, which has a size $\le 8r$,
by Lemma \ref{Hwpm-Loo},
\begin{align*}
\frac{\|\partial_xv_{n,k}(\cdot, t)\|_{L^2}}{x_{n,k}(t)}
&\le \|\phi'\|_{L^2} + \epsilon + 2(r/A)^k\sqrt{2r}\sum \sup|\partial_xHw_\pm(\cdot, t)|\\
&\le \|\phi'\|_{L^2} + C_5(r, \epsilon)
\end{align*}
where $C_5(r, \epsilon) = \epsilon + 2\sqrt{2r}(C_1(r, \epsilon) + 32C_2(r, \epsilon)/(7 - 14r))$.
\end{proof}

\subsection{Energy estimates}\label{1Denergy}
Now we are ready for the energy estimates of the profiles.
More specifically, we bound the difference between the evolved profile and its initial value.

\subsubsection{$H^1$ Energy estimate}
We first show that $W_{n,k}(\cdot, t)$ and $\phi$ are close in $H^1$. Let
\[
E_{n,k} = \int (\partial_xW_{n,k}(x, t) - \phi'(x))^2dx.
\]
Then $E_{n,k}(0) = 0$ and $dE_{n,k}(t)/dt = E_1 + E_2 + E_3 + E_4$, where
\begin{align*}
E_1 &= 2\int \partial_xv_{n,k}(x, t)(W_{n,k}(x, t) - \phi(x))(\partial_xW_{n,k}(x, t) - \phi'(x))dx,\\
E_2 &= 2\int \partial_xv_{n,k}(x, t)\phi(x)(\partial_xW_{n,k}(x, t) - \phi'(x))dx,\\
E_3 &= -2\int V_{n,k}(x, t)(\partial_xW_{n,k}(x, t) - \phi'(x))(\partial_x^2W_{n,k}(x, t) - \phi''(x))\\
&= \int v_{n,k}(x, t)(\partial_xW_{n,k}(x, t) - \phi'(x))^2dx,\\
E_4 &= 2\int V_{n,k}(x, t)(\partial_xW_{n,k}(x, t) - \phi'(x))\phi''(x)dx.
\end{align*}
By Lemma \ref{v-Loo},
\begin{align*}
|E_1| &\le 2\sqrt{E_{n,k}}\|\partial_xv_{n,k}(\cdot, t)\|_{L^2}\sup|W_{n,k}(\cdot, t) - \phi|\\
&\le 2E_{n,k}x_{n,k}(t)\sqrt{2r}(\|\phi'\|_{L^2} + C_5(r, \epsilon))/\pi,\\
|E_2| &\le 2\sqrt{E_{n,k}}\|\partial_xv_{n,k}(\cdot, t)\|_{L^2}\max|\phi|\\
&\le 2\sqrt{E_{n,k}}x_{n,k}(t)\max|\phi|(\|\phi'\|_{L^2} + C_5(r, \epsilon)),\\
|E_3| &\le E_{n,k}\sup|v_{n,k}(\cdot, t)| \le E_{n,k}x_{n,k}(t)(\max|H\phi| + C_3(r, \epsilon)),\\
|E_4| &\le 2\sqrt{E_{n,k}}\|\phi''\|_{L^2}\sup|V_{n,k}(\cdot, t)|\\
&\le 2\sqrt{E_{n,k}}x_{n,k}(t)\|\phi''\|_{L^2}(\max|\Phi| + C_4(r, \epsilon)).
\end{align*}

Now
\begin{align*}
\left| \frac{dE_{n,k}}{dt} \right|
&\le E_{n,k}x_{n,k}(t)(\max|H\phi| + C_3(r, \epsilon)\\
&+ 2\sqrt{2r}(\|\phi'\|_{L^2}/\pi + C_5(r, \epsilon))/\pi)\\
&+ 2\sqrt{E_{n,k}}x_{n,k}(t)(\max|\phi|(\|\phi'\|_{L^2} + C_5(r, \epsilon))\\
&+ \|\phi''\|_{L^2}(\max|H\phi| + C_4(r, \epsilon)))\\
&\le x_{n,k}(t)(C_6(r, \epsilon, \phi)E_{n,k} + C_7(r, \epsilon, \phi)\sqrt{E_{n,k}})
\end{align*}
so
\[
\left| \frac{d\sqrt{E_{n,k}}}{dt} \right|
\le x_{n,k}(t)(C_6(r, \epsilon, \phi)\sqrt{E_{n,k}} + C_7(r, \epsilon, \phi))/2.
\]
From the initial data $E(0) = 0$, we solve (for $t < 0$)
\begin{align*}
\sqrt{E_{n,k}}
&\le \frac{C_7(r, \epsilon, \phi)}{2}\int_t^0 x_{n,k}(s)e^{\frac{C_6(r, \epsilon, \phi)}{2}\int_t^s x_{n,k}(r)dr}ds\\
&\le C_7(r, \epsilon, \phi)Ie^{C_6(r, \epsilon, \phi)I/2}/2
\end{align*}
where
\[
I = \int_t^0 x_{n,k}(s)ds.
\]

Now we estimate the integral of $x_{n,k}(t)$. Since
\[
\dot x_{n,k}(t) = x_{n,k}(t)\sum_{j=0}^{k-1} x_{n,j}(t)HW_{n,j}(0, t)
\]
and $|HW_{n,j}(0, t) - 1| \le c(r)\epsilon$,
by Lemma \ref{int-le}, there is $T < 0$ such that for $t \in [T, 0]$,
\[
I \le \int_T^0 x_{n,k}(t)dt < \frac{2(A - 1)}{1 - c(r)\epsilon}.
\]
Then we can choose $A > 1$ depending on $r$, $\epsilon$ and $\phi$ so that
$\sqrt{E_{n,k}} < \epsilon$ to close the bootstrap assumption on the $H^1$ norm.

To close the bootstrap assumption on the support, note that
the support is initially $[-1 - r, -1 + r] \cup [1 - r, 1 + r]$,
so it suffices to show that the endpoints move by a distance less than $r$. Since
\[
\partial_tW_{n,k}(x, t) = v_{n,k}(x, t)W_{n,k}(x, t) - V_{n,k}(x, t)\partial_xW_{n,k}(x, t),
\]
the endpoints move at the speed of
\[
|V_{n,k}(x, t)| \le x_{n,k}(t)(\max|\Phi| + C_4(r, \epsilon)),
\]
the endpoints move by at most
\[
(\max|\Phi| + C_4(r, \epsilon))I
\le \frac{2(\max|\Phi| + C_4(r, \epsilon))(A - 1)}{1 - c(r)\epsilon}
\]
so we can choose $A > 1$ depending on $r$, $\epsilon$ and $\phi$ so that
the right hand side is less than $r$.
This way we close the bootstrap assumption on the support.

\subsubsection{Higher energy estimates}
Next we control higher derivatives of $W_{n,k}(\cdot, t)$.
\begin{lem}\label{Hwpm-Ck}
For $N \ge 2$, there are constants $C_i(N, r, \epsilon)$ (i = 1, 2, 3) such that on the support of $w_{n,k}(\cdot, t)$,
\begin{align*}
|\partial_x^NHw_-(x, t)|
%&\le N!r^{N-1}C_1(r, \epsilon)(A/r)^{Nk}/x_{n,k}(t)^{N-1}\\
&\le C_1(N, r, \epsilon)(A/r)^{Nk}/x_{n,k}(t)^{N-1},\\
|\partial_x^NHw_+(x, t)| 
%&\le \frac{(4/\sqrt7)^{N-1}(N + 1)!C_2(r, \epsilon)(A/r)^{Nk}}{(1 - 2r)^{N-1}x_{n,k}(t)^{N-1}}\\
&\le C_2(N, r, \epsilon)(A/r)^{Nk}/x_{n,k}(t)^{N-1},\\
\|\partial_x^N(v_{n,k} - x_{n,k}(t)W_{n,k})(\cdot, t)\|_{L^2} &\le C_3(N, r, \epsilon)x_{n,k}(t).
\end{align*}
\end{lem}
\begin{proof}
Similarly to Lemma \ref{Hwpm-Loo}, we have that, for $N \ge 2$ and $x \in [0, 1 - 2r]$,
\[
|\partial_x^NHW_{n,k}(\pm x, t)|
%\le (N - 1)!\frac{1 + \pi(1 - 2r)c(r)\epsilon}{\pi(1 - 2r - x)^{N+1}}
\le C(N)\frac{1 + \pi(1 - 2r)c(r)\epsilon}{(1 - 2r - x)^{N+1}}
\]
where $C(N)$ denotes a constant depending only on $N$.
Then for $j < k$ and $x \in \supp w_{n,k}(\cdot, t)$,
\begin{align*}
|\partial_x^NHw_{n,j}(x, t)|
%&\le N!\frac{(1 + \pi(1 - 2r)c(r)\epsilon)(A/r)^{Nj}}{\pi(1 - 3r - 2r^2)^{N+1}x_{n,j}(t)^{N-1}}\\
&\le C(N)\frac{(1 + \pi(1 - 2r)c(r)\epsilon)(A/r)^{Nj}}{(1 - 3r - 2r^2)^{N+1}x_{n,j}(t)^{N-1}}\\
%&\le N!\frac{(1 + \pi(1 - 2r)c(r)\epsilon)A^{j+(N-1)k}}{\pi(1 - 3r - 2r^2)^{N+1}r^{Nj}x_{n,k}(t)^{N-1}}.
&\le C(N)\frac{(1 + \pi(1 - 2r)c(r)\epsilon)A^{j+(N-1)k}}{(1 - 3r - 2r^2)^{N+1}r^{Nj}x_{n,k}(t)^{N-1}}.
\end{align*}
Summing over $j = 0, \dots, k - 1$ we get
\begin{align*}
|\partial_x^NHw_-(x, t)|
%&\le N!\frac{(1 + \pi(1 - 2r)c(r)\epsilon)A^{Nk-1}}{\pi(1 - 3r - 2r^2)^{N+1}r^{N(k-1)}(1 - r^N/A)x_{n,k}(t)^{N-1}}\\
&\le C(N)\frac{(1 + \pi(1 - 2r)c(r)\epsilon)A^{Nk-1}}{(1 - 3r - 2r^2)^{N+1}r^{N(k-1)}(1 - r^N/A)x_{n,k}(t)^{N-1}}\\
%&\le N!r^{N-1}C_1(r, \epsilon)(A/r)^{Nk}/x_{n,k}(t)^{N-1}.
&\le C_1(N, r, \epsilon)(A/r)^{Nk}/x_{n,k}(t)^{N-1}.
\end{align*}
For $k < j \le n$ we have that, for $|x| > 1 + 2r$,
\begin{align*}
|\partial_x^NHW_{n,j}(x, t)|
%&\le N!\frac{1 + \pi(1 - 2r)c(r)\epsilon}{2\pi}\\
&\le C(N)(1 + \pi(1 - 2r)c(r)\epsilon)\\
&\times\left( \frac{1}{(|x| - 1 - 2r)^{N+1}} - \frac{1}{(|x| + 1 + 2r)^{N+1}} \right)\\
%&\le \frac{(N + 1)!(1 + \pi(1 - 2r)c(r)\epsilon)(1 + 2r)(|x| + 1 + 2r)^N}{\pi(x^2 - (1 + 2r)^2)^{N+1}}\\
&\le C(N)\frac{(1 + \pi(1 - 2r)c(r)\epsilon)(1 + 2r)(|x| + 1 + 2r)^N}{(x^2 - (1 + 2r)^2)^{N+1}}\\
%&= \frac{(N + 1)!(1 + \pi(1 - 2r)c(r)\epsilon)(1 + 2r)}{\pi(|x| + 1 + 2r)(|x| - (1 + 2r))^{N+1}}
&= C(N)\frac{(1 + \pi(1 - 2r)c(r)\epsilon)(1 + 2r)}{(|x| + 1 + 2r)(|x| - (1 + 2r))^{N+1}}
\end{align*}
so for $x \in \supp w_{n,k}(\cdot, t)$,
\begin{align*}
|\partial_x^NHw_{n,j}(x, t)|
%&\le \frac{(N + 1)!(1 + \pi(1 - 2r)c(r)\epsilon)(1 + 2r)(A/r)^{Nj}}{%(\frac{A^j|x|}{x_{n,j}(t)r^j} + 1 + 2r)(\frac{A^j|x|}{x_{n,j}(t)r^j} - (1 + 2r))^{N+1}x_{n,j}(t)^{N-1}}\\
&\le C(N)\frac{(1 + \pi(1 - 2r)c(r)\epsilon)(1 + 2r)(A/r)^{Nj}}{(\frac{A^j|x|}{x_{n,j}(t)r^j} + 1 + 2r)(\frac{A^j|x|}{x_{n,j}(t)r^j} - (1 + 2r))^{N+1}x_{n,j}(t)^{N-1}}\\
%&\le \frac{(4/\sqrt7)^{N+1}(N + 1)!(1 + \pi(1 - 2r)c(r)\epsilon)(1 + 2r)(A/r)^{Nj}}{%|\frac{A^j|x|}{x_{n,j}(t)r^j}|^{N+2}x_{n,j}(t)^{N-1}}\\
&\le C(N)\frac{(1 + \pi(1 - 2r)c(r)\epsilon)(1 + 2r)(A/r)^{Nj}}{|\frac{A^j|x|}{x_{n,j}(t)r^j}|^{N+2}x_{n,j}(t)^{N-1}}\\
%&\le \frac{(4/\sqrt7)^{N+1}(N + 1)!(1 + \pi(1 - 2r)c(r)\epsilon)(1 + 2r)(A/r)^{Nj}}{%(1 - 2r)^3r^{3(k-j)}(A^j|x|/r^j)^{N-1}}\\
&\le C(N)\frac{(1 + \pi(1 - 2r)c(r)\epsilon)(1 + 2r)(A/r)^{Nj}}{(1 - 2r)^3r^{3(k-j)}(A^j|x|/r^j)^{N-1}}\\
%&\le \frac{(4/\sqrt7)^{N+1}(N + 1)!(1 + \pi(1 - 2r)c(r)\epsilon)(1 + 2r)(A/r)^{Nj}}{%(1 - 2r)^{N+2}r^{3(k-j)}(A/r)^{(N-1)(j-k)}x_{n,k}(t)^{N-1}}\\
&\le C(N)\frac{(1 + \pi(1 - 2r)c(r)\epsilon)(1 + 2r)(A/r)^{Nj}}{(1 - 2r)^{N+2}r^{3(k-j)}(A/r)^{(N-1)(j-k)}x_{n,k}(t)^{N-1}}\\
%&= \frac{(4/\sqrt7)^{N+1}(N + 1)!(1 + \pi(1 - 2r)c(r)\epsilon)(1 + 2r)(A/r)^{j+(N-1)k}}{%(1 - 2r)^{N+2}r^{3(k-j)}x_{n,k}(t)^{N-1}}.
&= C(N)\frac{(1 + \pi(1 - 2r)c(r)\epsilon)(1 + 2r)(A/r)^{j+(N-1)k}}{(1 - 2r)^{N+2}r^{3(k-j)}x_{n,k}(t)^{N-1}}.
\end{align*}
Summing over $j > k$ we get
\begin{align*}
|\partial_x^NHw_+(x, t)|
%&\le \frac{(4/\sqrt7)^{N+1}(N + 1)!(1 + \pi(1 - 2r)c(r)\epsilon)(1 + 2r)Ar^2(A/r)^{Nk}}{%(1 - 2r)^{N+2}(1 - Ar^2)x_{n,k}(t)^{N-1}}\\
&\le C(N)\frac{(1 + \pi(1 - 2r)c(r)\epsilon)(1 + 2r)Ar^2(A/r)^{Nk}}{(1 - 2r)^{N+2}(1 - Ar^2)x_{n,k}(t)^{N-1}}\\
%&\le \frac{(4/\sqrt7)^{N-1}(N + 1)!C_2(r, \epsilon)(A/r)^{Nk}}{%(1 - 2r)^{N-1}x_{n,k}(t)^{N-1}}.
&\le C(N)\frac{C_2(N, r, \epsilon)(A/r)^{Nk}}{(1 - 2r)^{N-1}x_{n,k}(t)^{N-1}}.
\end{align*}
The third bound follows from the sum of the first two.
\end{proof}

Now we bound higher Sobolev norms of $W_{n,k}(\cdot, t)$ using its evolution equation
\[
\partial_tW_{n,k}(x, t) = v_{n,k}(x, t)W_{n,k}(x, t) - V_{n,k}(x, t)\partial_xW_{n,k}(x, t).
\]
For $N \ge 2$ let
\[
E_{N,n,k} = \int (\partial_x^NW_{n,k}(x, t))^2dx.
\]

\begin{lem}
for $N \ge 2$ and $t \in [T, 0]$ we have that
\[
E_{N,n,k}(t) \le C(N, r, \epsilon, \phi)
\]
uniformly in $n$ and $k$.
\end{lem}
\begin{proof}
By integration by parts, the Sobolev multiplication theorem and the Kato--Ponce inequality, we have that
\begin{align*}
\frac{dE_{N,n,k}}{dt} &\le C(N)E_{N,n,k}\sup|v_{n,k}(\cdot, t)|\\
&+ \sqrt{E_{N,n,k}}\|\partial_x^Nv_{n,k}(\cdot, t)\|_{L^2}\sup|W_{n,k}(\cdot, t)|\\
&+ \sqrt{E_{N,n,k}}\|\partial_x^{N-1}v_{n,k}(\cdot, t)\|_{L^2}\sup|\partial_xW_{n,k}(\cdot, t)|.
\end{align*}
We use Lemma \ref{v-Loo} to bound the sup of $v_{n,k}$, Lemma \ref{Hwpm-Ck} to bound the Sobolev norms of $v_{n,k}$ and
\[
\sup|\partial_x^jW_{n,k}(\cdot, t)|
\le \sqrt{2r}\|\partial_x^{j+1}W_{n,k}(\cdot, t)\|_{H^{j+1}}/\pi
\le \sqrt{2E_{j+1,n,k}r}/\pi,
\]
to bound (the $C^j$ norms of) $W_{n,k}$, to get
it follows that
\begin{align*}
\frac{dE_{N,n,k}}{dt} &\lesssim C(N)E_{N,n,k}x_{n,k}(t)(\max|\phi| + C_3(r, \epsilon) + \sqrt{2E_{1,n,k}r}/\pi)\\
&+ x_{n,k}(t)\sqrt{2E_{N,n,k}E_{N-1,n,k}E_{2,n,k}r}/\pi + x_{n,k}(t)\sqrt{E_{N,n,k}}\\
& \times (C_3(N, r, \epsilon)\sqrt{2E_{1,n,k}r} + C_3(N - 1, r, \epsilon)\sqrt{2E_{2,n,k}r})/\pi.
\end{align*}
For $N \ge 2$, this can be written schematically as
\[
\frac{d\sqrt{E_{N,n,k}}}{dt} \le C(N, r, \epsilon, \phi, E_{N-1,n,k}, \dots, E_{1,n,k})
(\sqrt{E_{N,n,k}} + 1)
\]
so by induction, we get a bound of the desired form, uniform in $n$ and $k$.
\end{proof}

\subsection{Convergence to the limit}\label{1Dlimit}
Now we show that $w_n$ converges to a solution as $n \to \infty$.

\begin{prop}
There is a subsequence of $\{w_n\}$ converging in $L^2$ and in $C^\infty$ on compact sets not containing 0 to a solution $w$ to the De Gregorio equation.
\end{prop}
\begin{proof}
Since $x_{n,k}(t) \le A^k$, from the system of ODEs it follows that $\{x_{n,k}(t)\}_n$ is an equicontinuous family,
so by a diagonalization argument, we can pass to a subsequence so that $x_{n,k}(t) \to x_k(t)$ uniformly for $t \in [T, 0]$ and all $k$.

From the above we know that all the $C^N$ norms of $W_{n,k}$ are bounded uniformly in $n$ and $k$.
Again by passing to a subsequence we can assume that for all $k$ and $N$, $W_{n,k} \to W_k^*$ in $C^N$.
Then $w_{n,k} \to w_k^*$ in $C^N$. Since
\[
\supp w_{n,k} \subset x_{n,k}(t)(r/A)^k\supp W_{n,k} \subset [-(1 + 2r)r^k, (1 + 2r)r^k],
\]
the partial sums
\[
w_n = \sum_{k=0}^n w_{n,k} \to w = \sum_{k=0}^\infty w_k^*
\]
in $C^N$ away from $(-1/N, 1/N)$. Since
\[
\|w_{n,k}\|_{L^2} = x_{n,k}(t)^{3/2}(r/A)^{k/2}\|W_{n,k}\|_{L^2} \le (A^2r)^{k/2}(\|\phi\|_{L^2} + 4r\epsilon/\pi)
\]
and $A^2r < 1$,
\[
\sum_{j>k} \|w_{n,k}\|_{L^2} \to 0 \tag{$k \to \infty$}
\]
uniformly in $n$, so $w_n \to w$ in $L^2$ as well. Then by spatial localization,
$Hw_n \to Hw$ in $L^2$ and in $C^N$ away from $(-2/N, 2/N)$.
Integrating in $x$ from the origin shows that $u_n \to u$ locally uniformly and
in $C^\infty$ on compact sets not containing 0, with $\partial_xu = Hw$.
Since for each $n$, $w_n$ satisfies the De Gregorio equation
$\partial_tw_n + u_n\partial_xw_n = w_nHw_n$, passing to the limit shows that
$w$ satisfies the De Gregorio equation classically away from 0.
The weak formulation $\int (w_n\partial_t\varphi + u_nw_n\partial_x\varphi + 2\varphi w_nHw_n)
= w_n\varphi|_{-T}^0$ also passes to the limit, showing that $w$ is also a distributional solution,
with the expected boundary values at $t = -T$ and $t = 0$.
\end{proof}

\subsection{Properties of the solution}\label{1Dprop}
Now we show that our solution blows up in the H\"older norm.

\subsubsection{H\"older continuity}
\begin{thm}
Our solution $w(x, t)$ is H\"older continuous in $x$ if $|t|$ is small enough.
\end{thm}
\begin{proof}
We bound the difference $w(x, t) - w(y, t)$.
There is nothing to prove if $w(x, t) = w(y, t) = 0$,
so we assume without loss of generality that $w(x, t) < 0$ (so that $x > 0$).
If $y < 0$ (so that $w(y, t) \ge 0)$, from the inequality
$a^s + b^s \le 2^{1-s}(a + b)^s$ for $a$, $b \ge 0$, $s \in [0, 1]$ it follows that
\[
\frac{|w(y, t) - w(x, t)|}{|x - y|^s}
\le 2^{1-s}\max\left( \frac{-w(x, t)}{x^s}, \frac{w(y, t)}{(-y)^s} \right)
\]
so we can assume without loss of generality that $y \ge 0$, so that $w(y, t) \le 0$.

Since $w(x, t) < 0$, there is an integer $k$ such that $x \in \supp w_{n,k}$.
If $y \notin \cup_j \supp w_{n,j}$, then $w(y, t) = 0$.
Let $z$ be one of the two endpoints of $\supp w_{n,k}$
that is on the same side of $x$ as $y$. Then $w(z, t) = 0$ and $|x - z| \le |x - y|$, so
\[
\frac{|w(y, t) - w(x, t)|}{|x - y|^s} \le \frac{|w(z, t) - w(x, t)|}{|x - z|^s}.
\]
Thus we can assume without loss of generality that $y \in \supp w_{n,j}$ for some integer $j$.
If $j \neq k$, then the $z$ chosen above is between $x$ and $y$, so
$|x - y| \ge |z - y|$ and $|x - z|$. Also, $|w(y, t) - w(x, t)| \le \max(|w(y, t)|, |w(x, t)|)
= \max(|w(y, t) - w(z, t)|, |w(z, t) - w(x, t)|$, so
\[
\frac{|w(y, t) - w(x, t)|}{|x - y|^s}
\le \max\left( \frac{|w(y, t) - w(z, t)|}{|z - y|^s}, \frac{|w(z, t) - w(x, t)|}{|x - z|^s} \right)
\]
so without loss of generality we can assume that $y \in \supp w_{n,k}$. Then
\begin{align*}
|w(y, t) - w(x, t)| &\le \|\partial_xw_{n,k}(\cdot, t)\|_{L^2}\sqrt{|x - y|}\\
&\le (\|\phi'\|_{L^2} + \epsilon)\sqrt{x_k(t)(A/r)^k|x - y|}.
\end{align*}

If $t < 0$  by Lemma \ref{int-ge},
\[
\liminf_{k\to\infty} \int_t^0 x_k(s)ds \ge a
\]
where $a > 0$ solves $a = A(1 - e^{-a})$.
Then as $k \to \infty$,
\[
x_k(t) \le x_k(0)e^{(-a+o(1))k} = A^ke^{(-a+o(1))k}% < e^{(A-1-a+o(1))k} \le e^{(-(A-1)/2+o(1))k}
\]
so
\[
|x - y| \le 4rx_k(t)(r/A)^k \le r^ke^{(-a+o(1))k}
\]
so for $s \in [0, 1/2]$ we have that
\[
|w(y, t) - w(x, t)| \le (\|\phi'\|_{L^2} + \epsilon)(A^2e^{(2-2s)(-a+o(1))}/r^{2s})^{k/2}|x - y|^s
\]
so $w(\cdot, t) \in C^{s-}$ where $A^2e^{-(2-2s)a} = r^{2s}$,
or $\ln A - (1 - s)a = s\ln r$, or
\[
s = \frac{a - \ln A}{a - \ln r}.
\]
Since $\ln A < A - 1 = A(1 - e^{-\ln A})$, $a > \ln A$ so $s > 0$.
\end{proof}

\subsubsection{Rate of the blow up}
\begin{thm}\label{rate}
There is a constant $C > 0$ such that for all $T \le t < 0$, $1/|Ct| \le \max|w(\cdot, t)| \le C/|t|$.
\end{thm}
\begin{proof}
From the ansatz it follows that
\[
\max|w(\cdot, t)| = \sup_k x_k(t)\max|W_k^*(\cdot, t)|
\le (\max\phi + \epsilon\sqrt{r/3})\sup_k x_k(t)
\]
so it suffices to show the same bound for $\sup_k x_k(t)$.

For the upper bound, since $x_k(t)$ is increasing in $t$, we use Lemma \ref{int-le} (with a slight rescaling in time) to get
\[
x_k(t) \le \frac{1}{|t|}\int_T^0 x_k(s)ds \le \frac{a}{(1 - c(r)\epsilon)|t|}.
\]

For the lower bound, we assume without loss of generality that $t \ge -1$ and
let $k$ be such that $-A^{-k} \le t < -A^{-(k+1)}$. Then
\[
\frac{d\ln x_k(t)}{dt} = \sum_{j=0}^{k-1} x_j(t)HW_j(0, t)
\le (1 + c(r)\epsilon)\sum_{j=0}^{k-1} A^j < (1 + c(r)\epsilon)\frac{A^k}{A - 1}
\]
so
\[
\ln x_k(0) - \ln x_k(t) \le \frac{1 + c(r)\epsilon}{A - 1} =: C
\]
so $x_k(t) \ge e^{-C}A^k \ge e^{-C}/A|t|$.
\end{proof}

\subsubsection{Non-aymptotic self-similarity}\label{1DNonsim}
In Definition 1.1 of \cite{Hou3}, asymptotically self-similar blow up solutions of the De Gregorio equation are thus defined:
\begin{df}
A solution $w$ of the De Gregorio equation is asymptotically self-similar
if there is a profile $W$ in some weighted $H^1$ space such that
\[
(c_lx + aU)W_x = (c_w + U_x)W. \quad U_x = HW
\]
with $c_l \in \R$, $c_w < 0$ and that for some time-dependent scaling factors $C_w(t)$, $C_l(t) > 0$,
\[
C_w(t)w(C_l(t)x, t) \to W(x)
\]
in some weighted $L^2$ norm as $t$ approaches the blow up time.
\end{df}

Our goal is to rule out the existence of such a profile,
with its associated parameters $c_w$, $c_l$, $C_w$ and $C_l$.
To do so, we restrict their possible values in a series of lemmas.

\begin{lem}\label{sgnU}
If for $x\neq 0$  $W \neq 0$ then $U(x) \neq 0$  and has the same sign as $x$.
\end{lem}
\begin{proof}
Note that the spacetime scaling preserves the sign,
so when $x > 0$, $W(x) < 0$, and vice versa. Now integrating
\[
HW(x) = \frac{1}{\pi}P.V.\int_\R \frac{W(y)dy}{x - y}
\]
from $U(0) = 0$ shows that
\[
U(x) = \frac{1}{\pi}\int_\R W(y)\ln|x - y|dy
= \frac{1}{\pi}\int_{-\infty}^0 W(y)(\ln|x - y| - \ln|x + y|)dy.
\]
For $x > 0$ (resp. $< 0$), the integrand is non-negative (resp. non-positive),
so is the integral. Besides, the integral is exactly 0 only when $W = 0$. 
\end{proof}

By definition the profile $W$ is locally in $H^1$, so it is continuous.
We first show that it does not vanish at any point (unless it is trivial).
\begin{lem}\label{W=0}
If $c_l \ge 0$, $a > 0$ and $W(x) = 0$ for some $x \neq 0$, then $W = 0$.
\end{lem}
\begin{proof}
In the ODE satisfied by $W$, we view $U$ as unrelated to $W$.
Then it is linear in $W$ with continuous coefficients.
Moveover by Lemma \ref{sgnU}, if $x \neq 0$ then $c_lx + aU(x) \neq 0$,
so the ODE can be integrated directly to show that $W = 0$.
\end{proof}

We also add the mild requirement that $C_l$ is continuous, as is evident
from the more detailed account of the dynamic rescaling formulation in section 3 of \cite{Hou3}.
We can now make use of the vanishing of our solution on infinitely many intervals
that are closer and closer to the origin to show that the limiting profile $W$,
if there is one, must vanish identically. We first deal with the case when $c_l \ge 0$.
\begin{lem}\label{not-similar-focusing}
Our blow up solution does not tend to an asymptotically self-similar profile with $c_l \ge 0$.
\end{lem}
\begin{proof}
Assume that it is. We first show that $C_l(t) \to 0$ as $t \to 0$.
Since $w(x, t) = 0$ for $|x| > 1 + 2r$, $C_w(t)w(C_l(t)x, t) = 0$ for $|x| > (1 + 2r)/C_l(t)$.
If $C_l(t) \not \to 0$, then we can find a sequence $t_n \to 0$ and a constant $c > 0$ such that $C_l(t_n) > c$.
Then for $|x| > (1 + 2r)/c > (1 + 2r)/C_l(t_n)$ we have $0 = C_w(t_n)w(C_l(t_n)x, t_n) \to W(x)$ locally in $L^2$,
so $W(x) = 0$ for $|x| > (1 + 2r)/c$. By Lemma \ref{W=0}, $W(x) = 0$ for all $x \neq 0$, which is a contradiction.

Now we know that $C_l(t) \to 0$ as $t \to 0$. Since
\[
x \in [(1 + 2r)x_{k+1}(t), (1 - 2r)x_k(t)] \implies w(x, t) = 0,
\]
it follows that
\[
x \in [(1 + 2r)x_{k+1}(t)/C_l(t), (1 - 2r)x_k(t)/C_l(t)] \implies C_w(t)w(C_l(t)x, t) = 0.
\]
Pick $k_1$ such that $(1 + 2r)r^{k_1+1} \le C_l(T)$.
Then $(1 + 2r)x_{k_1+1}(T)/C_l(T) < 1$, but as $t \to 0$, $(1 + 2r)x_{k_1+1}(t)/C_l(t) \to \infty$
because $x_{k_1+1}(t)$ is bounded below away from 0. Then we can find $t_1 \in (T, 0)$ such that $(1 + 2r)x_{k_1+1}(t_1)/C_l(t_1) = 1$. If we started with $k_2$ such that $(1 + 2r)r^{k_2+1} \le \min_{[T,t_1/2]} C_l$, we would end up with $t_2 \in (t_1/2, 0)$ such that $(1 + 2r)x_{k_2+1}(t_2)/C_l(t_2) = 1$. By induction we get a sequence $k_n$ and another sequence $t_n \to 0$ such that $(1 + 2r)x_{k_n+1}(t_n)/C_l(t_n) = 1$. Since $x_{k+1}(t)/x_k(t) \le r$ (by Lemma \ref{monotone}), $(1 - 2r)x_{k_n}(t_n)/C_l(t_n) \ge (1 - 2r)/r(1 + 2r) \ge 4/3$ (using $r \le 1/4$), so $C_w(t_n)w(C_l(t_n)x, t_n)$ vanishes on $[1, 4/3]$, so does the $L^2$ limit $W$, which is again a contradiction as above.
\end{proof}

The case when $c_l < 0$ is easier, if we also add the requirement that in this case,
$C_l(t) \to \infty$ as $t \to 0$, which also follows from section 3 of \cite{Hou3}.
\begin{thm}\label{not-similar}
Our blow up solution is not asymptotically self-similar.
\end{thm}
\begin{proof}
By Lemma \ref{not-similar-focusing}, we can assume $c_l < 0$ and so $C_l(t) \to \infty$ as $t \to 0$.
Since $w(x, t) = 0$ for $|x| > 1 + 2r$, $C_w(t)w(C_l(t)x, t) = 0$ for $|x| > (1 + 2r)/C_l(t)$.
As $t \to 0$, $C_l(t) \to \infty$, so taking the $L^2_{loc}$ limit shows that $W = 0$ for $x \neq 0$,
contradicting the non-triviality of the limiting profile.
\end{proof}

\subsection{The generalized De Gregorio equation}\label{1Dgen}
The corresponding result for the generalized De Gregorio equation
\[
\partial_tw + au\partial_xw = wHw, \quad \partial_xu = Hw \tag{$a > 0$}
\]
can be proven similarly. The ansatz is now
\[
w_n(x, t) = \sum_{k=0}^n w_{n,k}(x, t)
\]
where
\[
w_{n,k}(x, t) = x_{n,k}(t)W_{n,k}\left( \frac{A^{ak}x}{x_{n,k}(t)^ar^k}, t \right)
\]
with the same initial data. From the evolution equation
\[
\partial_tw_{n,k} + u_n\partial_xw_{n,k} = w_{n,k}Hw_n,\quad
\partial_xu_n = Hw_n
\]
for $w_{n,k}$ we get the evolution equation
\begin{align*}
x_{n,k}(t)\partial_tW_{n,k}(\cdots, t)
&= (Hw_n(x, t) - \dot x_{n,k}(t)/x_{n,k}(t))w_{n,k}(x, t)\\
&- a(U_n(x, t) - x\dot x_{n,k}(t)/x_{n,k}(t))\partial_xw_{n,k}(x, t)
\end{align*}
for $W_{n,k}$. The ODE system is the same as before, so are the decomposition
$Hw_n = Hw_{n,k} + Hw_-+ Hw_+$ and the definitions of $v_{n,k}$ and $V_{n,k}$,
and the equation above becomes
\[
\partial_tW_{n,k}(x, t) = v_{n,k}(x, t)W_{n,k}(x, t) - aV_{n,k}(x, t)\partial_xW_{n,k}(x, t).
\]
Its time derivative is then
\begin{align*}
\partial_{tx}W_{n,k}(x, t) &= \partial_xv_{n,k}(x, t)W_{n,k}(x, t)\\
&+ (1 - a)v_{n,k}(x, t)\partial_xW_{n,k}(x, t) - aV_{n,k}(x, t)\partial_xW_{n,k}(x, t).
\end{align*}
The energy estimate has two extra terms $E_5 = 2(1 - a)E_3$ and
\[
E_6 = 2(1 - a)\int v_{n,k}(x, t)\phi'(x)(\partial_xW_{n,k}(x, t) - \phi'(x))dx
\]
with
\[
|E_6| \le 2|1 - a|\sqrt{E_{n,k}}x_{n,k}(t)\|\phi'\|_{L^2}(\max|H\phi| + C_3(r, \epsilon))
\]
which has the same type of estimates as $E_4$,
so the same energy estimate as before leads to
\[
\left| \frac{d\sqrt{E_{n,k}}}{dt} \right|
\le (1 + 2|1 - a|)x_{n,k}(t)(C_8(r, \epsilon, \phi)\sqrt{E_{n,k}} + C_7(r, \epsilon, \phi))/2
\]
and
\begin{align*}
\sqrt{E_{n,k}}
\le (1 + 2|1 - a|)C_7(r, \epsilon, \phi)Ie^{(1 + 2|1 - a|)C_8(r, \epsilon, \phi)I/2}/2
\end{align*}
where
\[
I = \int_t^0 x_{n,k}(s)ds \le c(r, \epsilon)(A - 1).
\]
Thus if both $A - 1$ and $a(A - 1)$ are sufficienly small, the same proof goes through,
so do the higher energy estimates, and by the same limiting argument we get a blow up in the sup norm from $C^s$ initial data,
where the dependence of $s \sim (A - 1)/|\ln r|$ on $a$ is that both $s$ and $as$ need to be sufficiently small.

The proof of the blow up rate and non-asymptotic self-similarity is the same as before.

\section{The 3D axi-symmetric Euler equation}
In this part we specify the ansatz of the blow up in subsection \ref{3Dansatz}. Estimates related to the profiles of the blow up are gathered in subsection \ref{3Dprofile}. Subsection \ref{3DHolder} is dedicated to H\"older estimates of the profiles. In subsection \ref{3Dsmooth} we smooth out the singularities on the plane $z = 0$, leaving the only one at the origin. Then we take the limit to get the actual blow up in subsection \ref{3Dlimit}. Subsection \ref{3Dprop} explores various aspects of the blow up, including its H\"older continuity, the rate of the blow up, and its non-asymptotic similarity.

\subsection{The ansatz}\label{3Dansatz}
In the introduction we have shown that if $w$ is axi-symmetric and odd with respect to $z$,
the velocity field it generates satisfies
\[
\partial_ru^r(0, 0) = \frac{3}{2}\int_0^\infty \int_0^\infty \frac{r^2zw(r, z)}{(z^2 + r^2)^{5/2}}drdz.
\]
Also, $u^r(0, z) = u^z(r, 0) = 0$, so $\partial_zu^r(0, 0) = \partial_ru^z(0, 0) = 0$.
If $w(r, z) = \phi(r)|z|^{1/12}\sgn z$, then
\[
\partial_ru^r|_{(0,0)}
= \frac{3}{2}\int_0^\infty r^2\phi(r)\int_0^\infty \frac{z^{13/12}dz}{(z^2 + r^2)^{5/2}}dr
= \frac{3B}{4}\int_0^\infty r^{-11/12}\phi(r)dr
\]
where $B$ equals the beta function value B$(25/24, 35/24)$.
We fix a smooth function $\phi \ge 0$ supported in $[1 - d, 1 + d]$ such that
\[
\frac{3B}{4}\int_0^\infty r^{-11/12}\phi(r)dr = 1.
\]

Let the vorticity $\omega = we_\theta$, where
\[
w_n(r, z, t) = \sum_{k=0}^n w_{n,k}(r, z, t)
\]
and
\[
w_{n,k}(r, z, t) = x_{n,k}(t)W_{n,k}\left( \frac{A^kr}{x_{n,k}(t)d^k}, \frac{x_{n,k}(t)^2z}{A^{2k}d^k}, t \right)
\]
where $x_{n,k}(t)$ are to be determined later. Then
\[
W_{n,k}(r, z, t) = \frac{w_{n,k}(r(d/A)^kx_{n,k}(t), z(A^2d)^k/x_{n,k}(t)^2, t)}{x_{n,k}(t)}.
\]
At time 0 we let
\[
x_{n,k}(0) = A^k,\quad W_{n,k}(r, z, 0) = A^{-k}w_{n,k}(rd^k, zd^k, 0) = \phi(r)\rho(z)|z|^{1/12}\sgn z
\]
where $\rho$ is even, equal to 1 on $[-Z, Z]$, between 0 and 1 elsewhere,
and supported in $[-2Z, 2Z]$. Then
\begin{align*}
\partial_tw_{n,k}(r, z, t) &= \dot x_{n,k}(t)W_{n,k}((r/x_{n,k}(t))(A/d)^k, zx_{n,k}(t)^2/(A^2d)^k, t)
\tag{the arguments of $W_{n,k}$ are the same in the next three lines}\\
&- r(\dot x_{n,k}(t)/x_{n,k}(t))(A/d)^k\partial_rW_{n,k}(\cdots)\\
&+ 2x_{n,k}(t)^2\dot x_{n,k}(t)z/(A^2d)^k\partial_zW_{n,k}(\cdots)\\
&+ x_{n,k}(t)\partial_tW_{n,k}(\cdots)\\
&= (\dot x_{n,k}(t)/x_{n,k}(t))w_{n,k}(r, z, t)\\
&- r(\dot x_{n,k}(t)/x_{n,k}(t))\partial_rw_{n,k}(r, z, t)\\
&+ 2z(\dot x_{n,k}(t)/x_{n,k}(t))\partial_zw_{n,k}(r, z, t)\\
&+ x_{n,k}(t)\partial_tW_{n,k}((r/x_{n,k}(t))(A/d)^k, zx_{n,k}(t)^2/(A^2d)^k, t).
\end{align*}
Then
\begin{align*}
0 &= r(\partial_t + u_n^r\partial_r + u_n^z\partial_z)(w_{n,k}(r, z, t)/r)\\
&= x_{n,k}(t)\partial_tW_{n,k}((r/x_{n,k}(t))(A/d)^k, zx_{n,k}(t)^2/(A^2d)^k, t)\\
&+ (\dot x_{n,k}(t)/x_{n,k}(t) - u_n^r(r, z, t)/r)w_{n,k}(r, z, t)\\
&- (r\dot x_{n,k}(t)/x_{n,k}(t) - u_n^r(r, z, t))\partial_rw_{n,k}(r, z, t)\\
&+ (2z\dot x_{n,k}(t)/x_{n,k}(t) + u_n^z(r, z, t))\partial_zw_{n,k}(r, z, t)
\end{align*}
so
\begin{align*}
\partial_tW_{n,k}
&=- \left( \frac{\dot x_{n,k}(t)}{x_{n,k}(t)} - \frac{u_n^r(r(d/A)^kx_{n,k}(t), z(A^2d)^k/x_{n,k}(t)^2, t)}{r(d/A)^kx_{n,k}(t)} \right)W_{n,k}\\
&+ \left( r\frac{\dot x_{n,k}(t)}{x_{n,k}(t)} - \frac{u_n^r(r(d/A)^kx_{n,k}(t), z(A^2d)^k/x_{n,k}(t)^2, t)}{(d/A)^kx_{n,k}(t)} \right)\partial_rW_{n,k}\\
&- \left( 2z\frac{\dot x_{n,k}(t)}{x_{n,k}(t)} + \frac{u_n^z(r(d/A)^kx_{n,k}(t), z(A^2d)^k/x_{n,k}(t)^2, t)}{(A^2d)^k/x_{n,k}(t)^2} \right)\partial_zW_{n,k}.
\end{align*}

We decompose $w_n = w_{n,k} + w_-+ w_+$, where
\[
w_- = \sum_{j<k} w_{n,j}, \quad w_+ = \sum_{j>k} w_{n,j}
\]
and similarly decompose $u_n = u_{n,k} + u_- + u_+$. We let
\begin{align*}
\frac{\dot x_{n,k}(t)}{x_{n,k}(t)}
&= \partial_ru_-^r(0, 0, t)
= \frac{3}{2}\int_0^\infty \int_0^\infty \frac{r^2zw_-(r, z, t)}{(z^2 + r^2)^{5/2}}drdz\\
&= \frac{3}{2}\sum_{j<k} x_{n,j}(t)\iint \frac{r^2zW_{n,j}((r/x_{n,j}(t))(A/d)^j, zx_{n,j}^2(t)/(A^2d)^j, t)}{(z^2 + r^2)^{5/2}}drdz\\
&= \frac{3}{2}\sum_{j<k} x_{n,j}(t)\int_0^\infty \int_0^\infty \frac{r^2zW_{n,j}(r, zx_{n,j}^3(t)/A^{3j}, t)}{(z^2 + r^2)^{5/2}}drdz,\\
v_{n,k}^r(r, z, t) &= \frac{u_n^r(r(d/A)^kx_{n,k}(t), z(A^2d)^k/x_{n,k}(t)^2, t)}{r(d/A)^kx_{n,k}(t)} - \partial_ru_-^r(0, 0, t),\\
v_{n,k}^z(r, z, t) &= \frac{u_n^z(r(d/A)^kx_{n,k}(t), z(A^2d)^k/x_{n,k}(t)^2, t)}{z(A^2d)^k/x_{n,k}(t)^2} - \partial_zu_-^z(0, 0, t)
\end{align*}
so
\[
\partial_tW_{n,k} - v_{n,k}^rW_{n,k} + rv_{n,k}^r\partial_rW_{n,k} + zv_{n,k}^z\partial_zW_{n,k} = 0
\]
or
\[
\partial_t(W_{n,k}/r) + rv_{n,k}^r\partial_r(W_{n,k}/r) + zv_{n,k}^z\partial_z(W_{n,k}/r) = 0.
\]

\subsection{The profiles}\label{3Dprofile}
Let $X(r, z, s, t)$ be the Lagrangian coordinates for $W_{n,k}/r$,
and $Y(r, z, s, t)$ be the Lagrangian coordinates for $w_n$, i.e.,
\begin{align*}
X(r, z, t, t) &= (r, z),\\
\partial_sX(r, z, s, t) &= (X^rv_{n,k}^r)(X(r, z, s, t), s)e_r + (X^zv_{n,k}^z)(X(r, z, s, t), s)e_z
\end{align*}
and the flow in the original coordinates is:
\begin{align*}
Y(r, z, t, t) &= (r, z),\\
\partial_sY(r, z, s, t) &= u_n^r(Y(r, z, s, t), s)e_r + u_n^z(Y(r, z, s, t), s)e_z.
\end{align*}
We assume, for $-T \le t \le s \le 0$ and $0 \le k \le n$,
\begin{enumerate}
    \item $w_{n,k}(r, z, t)$ is odd with respect to $z$.
    \item $1/(1 + d) \le X^r(r, z, 0, t)/r, X^z(r, z, 0, t)/z\le 1 + d$ ($z\neq0$).
    \item $|DY(\cdot, \cdot, s, t)|$ and $|DY(\cdot, \cdot, t, s)| \le 2A^{2k}/x_{n,k}(t)^2$
    on the support of $w_{n,k}(\cdot, \cdot, t)$, where we denote by $D$ the differential
    with respect to the variables $r$ and $z$.
\end{enumerate}
By symmetry, the first assumption is preserved under the flow. Now we check the remaining two.

\subsubsection{Estimates related to the profile}
\begin{lem}\label{HW-1}
For $K \in [0, 1]$ we have
\[
\left| \int_0^\infty \int_0^\infty \frac{r^2zW_{n,k}(r, Kz, t)}{(z^2 + r^2)^{5/2}}drdz - \frac{2K^{1/12}}{3} \right| \le c(d, Z)K^{1/12}
\]
where
\[
c(d, Z) = \frac{2((1 + d)^8 - 1)(1 + d)^{35/6}}{3(1 - d)^{35/6}}
+ \frac{32(1 + d)^{35/6}}{35BZ^{35/12}}.
\]
Note that as $d \to 0$ and $Z \to \infty$, $c(d, Z) \to 0$.
\end{lem}
\begin{proof}
Since $w_{n,k}/r$ is conserved along the flow, so is $W_{n,k}/r$ by the normalization of $W_{n,k}$.
Then $W_{n,k}(r, Kz, t) = rW_{n,k}(X(r, Kz, 0, t), 0)/X^r(r, Kz, 0, t)$, so the integral equals
\begin{align*}
&\int_0^\infty \int_0^\infty \frac{r^3zW_{n,k}(X(r, Kz, 0, t), 0)/X^r(r, Kz, 0, t)}{(z^2 + r^2)^{5/2}}drdz\\
= &\frac{1}{K^2}\int_0^\infty \int_0^\infty \frac{r^3zW_{n,k}(X(r, z, 0, t), 0)/X^r(r, z, 0, t)}{(z^2/K^2 + r^2)^{5/2}}drdz\\
= &\frac{1}{K^2}\int_0^\infty \int_0^\infty \frac{X^r(r, z, t, 0)^2X^z(r, z, t, 0)W_{n,k}(r, z, 0)}{(X^z(r, z, t, 0)^2/K^2 + X^r(r, z, t, 0)^2)^{5/2}}drdz\\
= &\frac{1}{K}\int_0^\infty \int_0^\infty \frac{X^r(r, Kz, t, 0)^2X^z(r, Kz, t, 0)\phi(r)\rho(Kz)(Kz)^{1/12}}{(X^z(r, Kz, t, 0)^2/K^2 + X^r(r, Kz, t, 0)^2)^{5/2}}drdz.
\end{align*}
Since
\[
\frac{1}{1 + d} \le \frac{X^z(r, Kz, t, 0)}{Kz}, \frac{X^r(r, Kz, t, 0)}{r} \le 1 + d,
\]
it follows that
\begin{align*}
&\left| \frac{X^r(r, Kz, t, 0)^2X^z(r, Kz, t, 0)\phi(r)\rho(Kz)(Kz)^{1/12}}{K(X^z(r, Kz, t, 0)^2/K^2 + X^r(r, Kz, t, 0)^2)^{5/2}}
- \frac{r^2z\phi(r)\rho(Kz)(Kz)^{1/12}}{(z^2 + r^2)^{5/2}} \right|\\
&\le ((1 + d)^8 - 1)\frac{r^2z\phi(r)\rho(Kz)(Kz)^{1/12}}{(z^2 + r^2)^{5/2}}.
\end{align*}
Since $\rho \le 1$ and $W_{n,k}(r, Kz, t)$ vanishes unless $r \ge (1 - d)/(1 + d) > (1 - d)^2$,
\begin{align*}
&\int_0^\infty \int_0^\infty \frac{r^2z|W_{n,k}(r, Kz, t) - \phi(r)\rho(Kz)(Kz)^{1/12}|}{(z^2 + r^2)^{5/2}}drdz\\
\le &((1 + d)^8 - 1)\int_0^\infty \int_0^\infty \frac{r^2\phi(r)(Kz)^{1/12} z}{(z^2 + (1 - d)^4)^{5/2}}drdz\\
= &\frac{B((1 + d)^8 - 1)}{2}\int_0^\infty \frac{r^2\phi(r)K^{1/12}}{(1 - d)^{35/6}}dr\\
\le &\frac{2((1 + d)^8 - 1)(1 + d)^{35/6}K^{1/12}}{3(1 - d)^{35/6}}
\end{align*}
where we have used the normalization on $\phi$ and $r \le (1 + d)^2$.

On the other hand,
\begin{align*}
\int_Z^\infty \int_0^\infty
\frac{r^2z^{13/12}\phi(r)\rho(Kz)}{(z^2 + r^2)^{5/2}}drdz
&\le \int_Z^\infty \int_0^\infty \frac{r^2z^{13/12}\phi(r)}{(z^2 + r^2)^{5/2}}drdz\\
&\le \int_0^\infty r^2\phi(r)\int_Z^\infty \frac{dz}{z^{47/12}}dzdr\\
&= \frac{12}{35Z^{35/12}}\int_0^\infty r^2\phi(r)dr\\
&\le \frac{16(1 + d)^{35/6}}{35BZ^{35/12}}.
\end{align*}
Since $\rho = 1$ on $[0, Z]$ and $K \in [0, 1]$,
\[
\int_0^Z \int_0^\infty \frac{r^2z\phi(r)\rho(Kz)(Kz)^{1/12}}{(z^2 + r^2)^{5/2}}drdz
= K^{1/12}\int_0^Z \int_0^\infty \frac{r^2z^{13/12}\phi(r)}{(z^2 + r^2)^{5/2}}drdz
\]
and by the normalization of $\phi$,
\[
\int_0^\infty \int_0^\infty \frac{r^2z^{13/12}\phi(r)}{(z^2 + r^2)^{5/2}}drdz = \frac{2}{3}
\]
so
\[
\left| \int_0^\infty \int_0^\infty \frac{r^2z\phi(r)\rho(Kz)(Kz)^{1/12}}{(z^2 + r^2)^{5/2}}drdz
- \frac{2K^{1/12}}{3} \right| \le \frac{32K^{1/12}(1 + d)^{35/6}}{35BZ^{35/12}}.
\]
Combining the two bounds shows the claim.
\end{proof}

\subsubsection{The ODE system}
Recall the ODE system
\[
\frac{\dot x_{n,k}(t)}{x_{n,k}(t)} = \partial_ru_-^r(0, 0, t).
\]
\begin{lem}\label{xk+1-ge}
If $d$ and $A - 1$ are sufficiently small and $Z$ is sufficiently large,
then there is $0 < a = a(d, Z, A) \le 4(2 - 3c(d, Z))(A - 1)/(2 - 15c(d, Z))$ such that
$x_{n,k+1}(t) \ge Ae^{-2a}x_{n,k}(t)$ for $t \in [-a, 0]$.
\end{lem}
\begin{proof}
By Lemma \ref{HW-1},
\begin{align*}
\frac{\dot x_{n,k}(t)}{x_{n,k}(t)}
&= \frac{3}{2}\sum_{j<k} x_{n,j}(t)\int_0^\infty \int_0^\infty \frac{r^2zW_{n,j}(r, zx_{n,j}^3(t)/A^{3j}, t)}{(z^2 + r^2)^{5/2}}drdz\\
&= \sum_{j<k} a_{n,j}(t)\frac{x_{n,j}(t)^{5/4}}{A^{j/4}}
\end{align*}
where $|a_{n,k}(t) - 1| \le 3c(d, Z)/2 \to 0$ as $d \to 0$ and $Z \to \infty$,
so we can make the right-hand side less than 1.
Let $X_{n,k} = x_{n,k}(t)^{5/4}/A^{k/4}$. Then
\[
\frac{\dot X_{n,k}(t)}{X_{n,k}(t)} = \frac54\sum_{j<k} a_{n,j}(t) X_{n,j}(t).
\]
By Lemma \ref{int-le} (applied with $b = (2 + 3c(d, Z))/(2 - 3c(d, Z))$,
there is $0 < a = a(d, Z, A) \le 4(2 - 3c(d, Z))(A - 1)/(2 - 15c(d, Z))$ such that
\[
\int_{-a}^0 X_{n,k}(t)dt \le a.
\]
% Moreover, for fixed $d$ and $Z$, $a(d, Z, A)$ remains bounded as $A \to 1+$.
Since
\[
\frac{d}{dt}(\ln x_{n,k+1}(t) - \ln x_{n,k}(t)) = a_{n,k}(t)X_{n,k}(t) < 2X_{n,k}(t)
\]
and $x_{n,k}(0) = A^k$, it follows that for $t \in [-a, 0]$,
\[
x_{n,k+1}(t) > Ae^{-2a}x_{n,k}(t).
\]
\end{proof}

\subsubsection{Evolution of the profiles}
We first bound the H\"older norms of $w_{n,k}$ (the scalar vorticity), $w_{n,k}e_\theta$ (the vorticity vector) and $\nabla u_{n,k}$ (the deformation tensor).
\begin{lem}\label{w-Holder}
\begin{align*}
\|w_{n,k}(\cdot, \cdot, t)\|_{\dot C^{1/12}}
&\le C(d, Z)x_{n,k}(t)^{5/6}A^{k/6} d^{-k/12},\\
\|(w_{n,k}e_\theta)(\cdot, \cdot, t)\|_{\dot C^{1/12}}
&\le 3C(d, Z)x_{n,k}(t)^{5/6}A^{k/6} d^{-k/12},\\
\|\nabla u_{n,k}(\cdot, \cdot, t)\|_{\dot C^{1/12}}
&\le 3CC(d, Z)x_{n,k}(t)^{5/6}A^{k/6} d^{-k/12},
\end{align*}
where
\(
C(d, Z) = 2^{1/12}(1 + d)^2(\|\phi(r)\rho(z)|z|^{1/12}\sgn z/r\|_{\dot C^{1/12}}
+ (2Z)^{1/12}(1 + d)^{11/6}\max|\phi(r)/r|
\)
and $C$ is the constant in the $\dot C^{1/12}$ estimates of the singular integral operator whose kernel is the derivative of the Biot--Savart kernel.
\end{lem}
\begin{proof}
Since $\tilde w_{n,k}(r, z, t) := w_{n,k}(r, z, t)/r$ is conserved along the flow,
\begin{align*}
\max|\tilde w_{n,k}(\cdot, \cdot, t)|
&= \max|\tilde w_{n,k}(\cdot, \cdot, 0)|
= (A/d)^k(2Z)^{1/12}\max|\phi(r)/r|,\\
\|\tilde w_{n,k}(\cdot, \cdot, t)\|_{\dot C^{1/12}}
&= \|\tilde w_{n,k}(Y(\cdot, \cdot, 0, t), 0)\|_{\dot C^{1/12}}\\
&\le \max|DY(\cdot, \cdot, 0, t)|^{1/12}\|\tilde w_{n,k}(\cdot, \cdot, 0)\|_{\dot C^{1/12}}\\
&\le 2^{1/12}\sqrt[6]{A^k/x_{n,k}(t)}A^kd^{-13k/12}\\
&\times \|\phi(r)\rho(z)|z|^{1/12}\sgn z/r\|_{\dot C^{1/12}}.
\end{align*}
Then
\begin{align*}
\|w_{n,k}(\cdot, \cdot, t)\|_{\dot C^{1/12}}
&\le \|\tilde w_{n,k}(\cdot, \cdot, t)\|_{\dot C^{1/12}}\max_{\supp w_{n,k}}r\\
&+ \max|\tilde w_{n,k}(\cdot, \cdot, t)|\|r\|_{\dot C^{1/12}(\supp w_{n,k})}\\
&\le 2^{1/12}\sqrt[6]{A^k/x_{n,k}(t)}A^kd^{-13k/12}(1 + d)^2x_{n,k}(t)d^k/A^k\\
&\times \|\phi(r)\rho(z)|z|^{1/12}\sgn z/r\|_{\dot C^{1/12}}\\
&+ (A/d)^k(2Z)^{1/12}((1 + d)^2x_{n,k}(t)d^k/A^k)^{11/12}\max|\phi(r)/r|\\
&= 2^{1/12}(1 + d)^2x_{n,k}(t)\sqrt[6]{A^k/x_{n,k}(t)}d^{-k/12}\\
&\times \|\phi(r)\rho(z)|z|^{1/12}\sgn z/r\|_{\dot C^{1/12}}\\
&+ (2Z)^{1/12}(1 + d)^{11/6}x_{n,k}(t)\sqrt[12]{A^k/x_{n,k}(t)}d^{-k/12}\\
&\times \max|\phi(r)/r|.
\end{align*}
Now the first bound follows from $A^k \ge x_{n,k}(t)$.

Now we turn to the first bound.
For $X = (r, \theta, z)$ and $Y = (r', \theta', z')$ in cylindrical coordinates,
assume without loss of generality that $r' \ge r$. Let $Z = (r, \theta', z)$.
Then $|X - Y| \ge |X - Z|$ and $|Y - Z|$, so
\begin{align*}
\frac{|(we_\theta)(X, t) - (we_\theta)(Y, t)|}{|X - Y|^{1/12}}
&\le \frac{|(we_\theta)(X, t) - (we_\theta)(Z, t)|}{|X - Z|^{1/12}}\\
&+ \frac{|(we_\theta)(Z, t) - (we_\theta)(Y, t)|}{|Z - Y|^{1/12}}.
\end{align*}
Since
\begin{align*}
|(we_\theta)(X, t) - (we_\theta)(Z, t)|
&= |w(r, z, t)||X - Z|/r \le 2|w(r, z, t)||X - Z|^{1/12}/r^{1/12}
\tag{using $|X - Z| \le 2r$}\\
&\le 2\|w(\cdot, \cdot, t)\|_{\dot C^{1/12}}|X - Z|^{1/12}
\end{align*}
and
$|(we_\theta)(Z, t) - (we_\theta)(Y, t)| = |w(r, z, t) - w(r', z', t)|$, the second bound follows.

For the third bound, we use the standard H\"older estimates on singular integral operators
and note that $\nabla_{r,z} u_{n,k}(r, z, t)$ is equivalent to $\nabla_{x,z} u_{n,k}(x, y, z, t)$.
\end{proof}

We now estimate $v_{n,k}^r$ and $v_{n,k}^z$.
\begin{lem}\label{gru--Holder}
If $d \le 1/4$ and $0 \le -t \le a = a(d, Z, A)$, then on the set
\[
S = \{r \le (1 + d)^2x_{n,k}(t)(d/A)^k, |z| \le 2Z(A^2d)^k/x_{n,k}(t)^2\}
\]
we have that
\[
\|\nabla u_-(\cdot, \cdot, t)\|_{\dot C^{1/12}(S)} < C_-(d, Z, A)x_{n,k}(t)^{5/6}A^{k/6}d^{-k/12}
\]
where
\[
C_-(d, Z, A) = \frac{4C(d, Z)d^{1/12} e^{2a}}{1 - d^{1/12} e^{2a}}\ln\frac{(1 + d)^3}{1 - 3d - d^2 - d^3}.
\]
Note that for fixed $d$ and $Z$, $C_-(d, Z, A)$ remains bounded as $A \to 1+$.
\end{lem}
\begin{proof}
From the expression of the Biot--Savart kernel it follows that
\[
|\nabla u_{n,j}(x, t) - \nabla u_{n,j}(x', t)|
\le \frac{1}{\pi}\iiint \frac{|w_{n,j}(x + y, t) - w_{n,j}(x' + y, t)|}{|y|^3}dy
\]
Since $r \le (1 + d)^2x_{n,k}(t)(d/A)^k$ and supp $w_{n,j}$ is contained in
\[
\{(1 - d)^2x_{n,j}(t)(d/A)^j \le r \le (1 + d)^2x_{n,j}(t)(d/A)^j\},
\]
we only need to take the integral over the region
\begin{align*}
&\{(1 - d)^2x_{n,j}(t)(d/A)^j - (1 + d)^2x_{n,k}(t)(d/A)^k \le r\\
&\le (1 + d)^2x_{n,j}(t)(d/A)^j + (1 + d)^2x_{n,k}(t)(d/A)^k\}.
\end{align*}
Since $j < k$, $x_{n,k}(t)/A^k \le x_{n,j}(t)/A^j$, so
\[
(1 + d)^2x_{n,k}(t)(d/A)^k \le (1 + d)^2x_{n,j}(t)d^k/A^j \le d(1 + d)^2x_{n,j}(t)(d/A)^j
\]
so we only need to take the integral over the region
\[
R = \{1 - 3d - d^2 - d^3 \le r(A/d)^j/x_{n,j}(t) \le (1 + d)^3\},
\]
i.e.,
\[
|\nabla u_{n,j}(x, t) - \nabla u_{n,j}(x', t)|
\le \frac{1}{\pi}\iiint_R \frac{\|w_{n,j}(\cdot, \cdot, t)\|_{\dot C^{1/12}}|x' - x|^{1/12}}{|y|^3}dy
\]
so
\[
\|\nabla u_{n,j}(r, z, t)\|_{\dot C^{1/12}} \le \frac{1}{\pi}\iiint_R \frac{\|w_{n,j}(\cdot, \cdot, t)\|_{\dot C^{1/12}}}{|y|^3}dy.
\]
By scaling symmetry,
\begin{align*}
\iiint_R \frac{dy}{|y|^3}
&= 2\pi\int_{1 - 3d - d^2 - d^3}^{(1 + d)^3} \int_{-\infty}^\infty \frac{rdzdr}{(r^2 + z^2)^{3/2}}
= 4\pi\int_{1 - 3d - d^2 - d^3}^{(1 + d)^3} \frac{dr}{r}\\
&= 4\pi\ln\frac{(1 + d)^3}{1 - 3d - d^2 - d^3}
\end{align*}
so
\[
\|\nabla u_{n,j}(\cdot, \cdot, t)\|_{\dot C^{1/12}}
\le 4\|w_{n,j}(\cdot, \cdot, t)\|_{\dot C^{1/12}}\ln\frac{(1 + d)^3}{1 - 3d - d^2 - d^3}.
\]
By Lemma \ref{w-Holder},
\[
\|\nabla u_{n,j}(\cdot, \cdot, t)\|_{\dot C^{1/12}}
\le 4C(d, Z)x_{n,j}(t)^{5/6}A^{j/6}d^{-j/12}\ln\frac{(1 + d)^3}{1 - 3d - d^2 - d^3}
\]
so summing over $j < k$ using $A^{k-j}x_{n,k}(t) \le x_{n,j}(t) < (e^{2a}/A)^{k-j}x_{n,k}(t)$ shows that
\[
\|\nabla u_-(\cdot, \cdot, t)\|_{\dot C^{1/12}}
< 4C(d, Z)\frac{(d^{1/12} e^{2a}/A)x_{n,k}(t)^{5/6}A^{k/6}}{(1 - d^{1/12} e^{2a}/A)d^{k/12}}\ln\frac{(1 + d)^3}{1 - 3d - d^2 - d^3}
\]
from which the result follows using $A > 1$.
\end{proof}

\begin{lem}\label{gru+-Holder}
If $d \le 1/4$ and $0 \le -t \le a = a(d, Z, A)$ then
\[
\|\nabla u_+(\cdot, \cdot, t)\|_{\dot C^{1/12}(\supp w_{n,k}(\cdot, \cdot, t))}
< C_+(d, Z, A)x_{n,k}(t)^{5/6}A^{k/6}d^{-k/12}
\]
where
\[
C_+(d, Z, A) = \frac{2C(d, Z)(1 + d)^6d^{23/12}Ae^{a/3}}{(1 - 3d - d^2 - d^3)^2(1 - d^{23/12}Ae^{a/3})}.
\]
Note that for fixed $d$ and $Z$, $C_+(d, Z, A)$ remains bounded as $A \to 1+$.
\end{lem}
\begin{proof}
As before,
\[
|\nabla u_{n,j}(x, t) - \nabla u_{n,j}(x', t)|
\le \frac{1}{\pi}\iiint \frac{|w_{n,j}(x + y, t) - w_{n,j}(x' + y, t)|}{|y|^3}dy
\]
For $j > k$, since $r > (1 - d)^2x_{n,k}(t)(d/A)^k$ and supp $w_{n,j}$ is contained in
\[
\{(1 - d)^2x_{n,j}(t)(d/A)^j \le r \le (1 + d)^2x_{n,j}(t)(d/A)^j\},
\]
the horizontal component of $y$ is at least
\[
(1 - d)^2x_{n,k}(t)(d/A)^k - (1 + d)^2x_{n,j}(t)(d/A)^j \ge (1 - 3d - d^2 - d^3)x_{n,k}(t)(d/A)^k
\]
and is at most
\[
(1 + d)^2x_{n,k}(t)(d/A)^k + (1 + d)^2x_{n,j}(t)(d/A)^j \le (1 + d)^3x_{n,k}(t)(d/A)^k.
\]
On this region
\begin{align*}
\iiint \frac{dy}{|y|^3}
&\le \int_{-\infty}^\infty \frac{\pi(1 + d)^6x_{n,j}(t)^2(d/A)^{2j}dz}{((1 - 3d - d^2 - d^3)^2x_{n,k}(t)^2(d/A)^{2k} + z^2)^{3/2}}\\
&= \frac{2\pi(1 + d)^6x_{n,j}(t)^2(d/A)^{2j-2k}}{(1 - 3d - d^2 - d^3)^2x_{n,k}(t)^2}
\le \frac{2\pi(1 + d)^6d^{2j-2k}}{(1 - 3d - d^2 - d^3)^2}
\end{align*}
so
\[
\|\nabla u_{n,j}(\cdot, \cdot, t)\|_{\dot C^{1/12}}
\le 2C(d, Z)(A^j/x_{n,j}(t))^{1/6}d^{-j/12}\frac{(1 + d)^6x_{n,j}(t)d^{2j-2k}}{(1 - 3d - d^2 - d^3)^2}
\]
so summing over $j > k$ using $(Ae^{-2a})^{j-k}x_{n,k}(t) < x_{n,j}(t) \le A^{j-k}x_{n,k}(t)$ shows that
\[
\|\nabla u_+(\cdot, \cdot, t)\|_{\dot C^{1/12}}
< 2C(d, Z)\frac{d^{23/12}Ae^{a/3}}{1 - d^{23/12}Ae^{a/3}}\frac{(1 + d)^6x_{n,k}(t)^{5/6}A^{k/6}}{(1 - 3d - d^2 - d^3)^2d^{k/12}}
\]
from which the result follows.
\end{proof}

\begin{lem}\label{gru+0}
If $d \le 1/4$ and $(r, 0)$ is in the support of $w_{n,k}(\cdot, \cdot, t)$ then
\[
|\nabla u_+(r, 0, t)| \le C_+^0(d, Z, A)x_{n,k}(t)
\]
where
\[
C_+^0(d, Z, A) = \frac{4(3d + d^3)(2Z)^{1/12}(d^2A)\max|\phi(r)/r|}{3(1 - 3d - d^2 - d^3)^2(1 - d^2A)}.
\]
Note that for fixed $d$ and $Z$, $C_+^0(d, Z, A)$ remains bounded as $A \to 1+$.
\end{lem}
\begin{proof}
Similarly to Lemma \ref{gru+-Holder}, for $j > k$,
\begin{align*}
|\nabla u_{n,j}(r, 0, t)|
&\le 2\iint_{(1-d)^2x_{n,j}(t)d^j/A^j}^{(1+d)^2x_{n,j}(t)d^j/A^j} \frac{r'|w_{n,j}(r', z, t)|dr'dz}{((1 - 3d - d^2 - d^3)^2x_{n,k}(t)^2(d/A)^{2k} + z^2)^{3/2}}\\
&\le \int_\R \frac{2A^j(2Z)^{1/12}dz}{d^j((1 - 3d - d^2 - d^3)^2x_{n,k}(t)^2(d/A)^{2k} + z^2)^{3/2}}\\
&\times \int_{(1-d)^2x_{n,j}(t)d^j/A^j}^{(1+d)^2x_{n,j}(t)d^j/A^j} r'^2\max|\phi(r)/r|dr'
\tag{$\max|w_{n,k}(r, z, t)/r| = (A/d)^k(2Z)^{1/12}\max|\phi(r)/r|$}\\
&\le \frac{2A^j(2Z)^{1/12}}{d^j(1 - 3d - d^2 - d^3)^2x_{n,k}(t)^2(d/A)^{2k}}\\
&\times \frac{x_{n,j}(t)^3d^{3j}(6d + 2d^3)}{3A^{3j}}\max|\phi(r)/r|\\
&= \frac{4(3d + d^3)(2Z)^{1/12}(d/A)^{2j-2k}x_{n,j}(t)^3}{3(1 - 3d - d^2 - d^3)^2x_{n,k}(t)^2}
\max|\phi(r)/r|\\
&\le \frac{4(3d + d^3)(2Z)^{1/12}(d^2A)^{j-k}x_{n,k}(t)}{3(1 - 3d - d^2 - d^3)^2}
\max|\phi(r)/r|
\end{align*}
where we have used $x_{n,j}(t) \le A^{j-k}x_{n,k}(t)$.
Summing over $j > k$ shows the claim.
\end{proof}

\subsection{H\"older estimates}\label{3DHolder}
Recall that
\[
(\partial_t + rv_{n,k}^r\partial_r + zv_{n,k}^z\partial_z)(W_{n,k}/r)
= (\partial_t + u_n^r\partial_r + u_n^z\partial_z)(w_{n,k}/r)
= 0.
\]
where
\begin{align*}
v_{n,k}^r(r, z, t) &= \frac{u_n^r(r(d/A)^kx_{n,k}(t), z(A^2d)^k/x_{n,k}(t)^2, t)}{r(d/A)^kx_{n,k}(t)} - \partial_ru_-^r(0, 0, t),\\
v_{n,k}^z(r, z, t) &= \frac{u_n^z(r(d/A)^kx_{n,k}(t), z(A^2d)^k/x_{n,k}(t)^2, t)}{z(A^2d)^k/x_{n,k}(t)^2} - \partial_zu_-^z(0, 0, t).
\end{align*}
Since $u_n^r = 0$ when $r = 0$ and $u_n^z = 0$ when $z = 0$,
\begin{align*}
|v_{n,k}^r(r, z, t)|
&\le \int_0^{r(d/A)^kx_{n,k}(t)} \frac{|\partial_ru_n^r(s, z(A^2d)^k/x_{n,k}(t)^2, t) - \partial_ru_-^r(0, 0, t)|}{r(d/A)^kx_{n,k}(t)}ds,\\
&\le \max_{s\in[0,r(d/A)^kx_{n,k}(t)]}|\partial_ru_n^r(s, z(A^2d)^k/x_{n,k}(t)^2, t) - \partial_ru_-^r(0, 0, t)|,\\
|v_{n,k}^z(r, z, t)|
&\le \max_{s\in[0,{z(A^2d)^k/x_{n,k}(t)^2}]} |\partial_zu_n^z(r(d/A)^kx_{n,k}(t), s, t) - \partial_zu_-^z(0, 0, t)|.
\end{align*}
When $0 \le -t \le a = a(d, Z, A)$,
\begin{align*}
\|\nabla u_-\|_{\dot C^{1/12}}
&< C_-(d, Z, A)x_{n,k}(t)^{5/6}A^{k/6}d^{-k/12},\\
\|\nabla u_{n,k}\|_{\dot C^{1/12}}
&\le 3CC(d, Z)x_{n,k}(t)^{5/6}A^{k/6}d^{-k/12},\\
\|\nabla u_+\|_{\dot C^{1/12}}
&< C_+(d, Z, A)x_{n,k}(t)^{5/6}A^{k/6}d^{-k/12},
\end{align*}
it follows that
\begin{align*}
\max|\nabla u_-(r, z, t) - \nabla u_-(0, 0, t)|
&< C_-(d, Z, A)((1 + d)^{1/6}x_{n,k}(t)^{11/12}A^{k/12}\\
&+ (2Z)^{1/12}x_{n,k}(t)^{2/3}A^{k/3}),\\
\max|\nabla u_{n,k}(r, z, t) - \nabla u_{n,k}(0, 0, t)|
&\le 3CC(d, Z)((1 + d)^{1/6}x_{n,k}(t)^{11/12}A^{k/12}\\
&+ (2Z)^{1/12}x_{n,k}(t)^{2/3}A^{k/3}),\\
\max|\nabla u_+(r, z, t) - \nabla u_+(r, 0, t)|
&< C_+(d, Z, A)((1 + d)^{1/6}x_{n,k}(t)^{11/12}A^{k/12}\\
&+ (2Z)^{1/12}x_{n,k}(t)^{2/3}A^{k/3}).
\end{align*}
By Lemma \ref{HW-1},
\[
|\nabla u_{n,k}(0, 0, t)| \le (1 + 3c(d, Z)/2)x_{n,k}(t)^{5/4}/A^{k/4}
\le (1 + 3c(d, Z)/2)x_{n,k}(t).
\]
By Lemma \ref{gru+0},
\[
|\nabla u_+(r, 0, t)| \le C_+^0(d, Z, A)x_{n,k}(t)
\]
so on the support of $W_{n,k}(\cdot, \cdot, t)$,
\begin{align*}
|v_{n,k}(r, z, t)| &< (3CC(d, Z) + C_-(d, Z, A) + C_+(d, Z, A))\\
&\times ((1 + d)^{1/6}x_{n,k}(t)^{11/12}A^{k/12} + (2Z)^{1/12}x_{n,k}(t)^{2/3}A^{k/3})\\
&+ (1 + 3c(d, Z)/2 + C_+^0(d, Z, A))x_{n,k}(t)\\
&\le C(d, Z, A)x_{n,k}(t)^{2/3}A^{k/3}
\end{align*}
where
\begin{align*}
C(d, Z, A) &= (3CC(d, Z) + C_-(d, Z, A) + C_+(d, Z, A))((1 + d)^{1/6} + (2Z)^{1/12})\\
&+ 1 + 3c(d, Z)/2 + C_+^0(d, Z, A).
\end{align*}
For fixed $d$ and $Z$, $C(d, Z, A)$ remains bounded as $A \to 1+$.

Recall that
\begin{align*}
\partial_sX(r, z, s, t) &= (X^rv_{n,k}^r)(X(r, z, s, t), s)e_r + (X^zv_{n,k}^z)(X(r, z, s, t), s)e_z,\\
\partial_sY(r, z, s, t) &= u_n^r(Y(r, z, s, t), s)e_r + u_n^z(Y(r, z, s, t), s)e_z.
\end{align*}
Since
\[
|\partial_sX^r(r, z, s, t)/X^r(r, z, s, t)| \le |v_{n,k}^r(X^r(r, z, s, t), s)|
\]
and a similar inequality holds for the $z$ component provided that $z \neq 0$,
it follows that
\begin{align*}
|\ln (X^r(r, z, 0, t)/r)| &, |\ln (X^z(r, z, 0, t)/z)|\le C(d, Z, A)\int_t^0 x_{n,k}(s)^{2/3}A^{k/3}ds.
\end{align*}
Also,
\[
\partial_s\ln|DY(r, z, s, t)|
\le \max_{\supp w_{n,k}} |\nabla u_n(\cdot, \cdot, s)|
\]
Since $\partial_ru_-^r(0, 0, s) = \dot x_{n,k}(t)/x_{n,k}(t)$ and $\partial_zu_-^z(0, 0, s) = -2\dot x_{n,k}(t)/x_{n,k}(t)$, the operator norm of $\nabla u_-(0, 0, s)$ is $2\dot x_{n,k}(t)/x_{n,k}(t)$.
For the other terms we use the above bounds. The result is that, for $t \le s \le 0$,
\begin{align*}
\ln |dY(r, z, s, t)| &, \ln |dY(r, z, t, s)|\\
&\le 2\ln(A^k/x_{n,k}(t)) + C(d, Z, A)\int_t^0 x_{n,k}(s)^{2/3}A^{k/3}ds.
\end{align*}
Thus, to close the remaining two bootstrap assumptions, it suffices that
\[
C(d, Z, A)\int_{-T}^0 x_{n,k}(s)^{2/3}A^{k/3}ds < 1 + d.
\]
The idea is to use Lemma \ref{int-ge} to get a lower bound of the time integral of $x_{n,k}$,
and then turn that into an upper bound of the integral above,
which we prove after a couple of lemmas.

\begin{lem}\label{int-ge2}
Let $\alpha \in (0, 1/2)$ and $A = e^{(3-2\alpha)c/4} > 1$.
Let $I_{n+1} = A(1 - e^{-I_n})$. Then for $n \ge 0$,
if $I_0 \ge cA^{-n}e^{2/(1-2\alpha)}$ then $I_n \ge c$.
\end{lem}
\begin{proof}
Since $(1 - e^{-x})/x$ decreases from 1 to 0 as $x$ increases from 0 to $\infty$
and $A > 1$, there is a unique positive root $x$ of $x = A(1 - e^{-x})$.
If $I_0 < x$ then $I_n$ increases in $n$ and tends to $x$.
If $I_0 > x$ then $I_n$ decreases in $n$ and tends to $x$.
Since $A(1 - e^{-c}) \ge Ace^{-c/2} = e^{(1-2\alpha)c/4}c > c$, $x > c$.
Thus if $I_0 \ge c$ then $I_n \ge c$ and we are done.

Now assume that $I_0 < c$. Then $I_n$ increases in $n$.
Also assume the contrary of the conclusion, i.e., $I_n < c$.
Then $I_k < c$ for all $k = 0, \dots, n$.
Since $I_{k+1} \ge AI_ke^{-I_k/2}$, $I_n \ge A^nI_0e^{-(I_0 + \cdots + I_{n-1})/2}$.
Since $I_{k+1} \ge AI_ke^{-c/2} = e^{(1-2\alpha)c/4}I_k$,
\begin{align*}
I_0 + \cdots + I_{n-1} &\le I_n(e^{-(1-2\alpha)c/4} + e^{-2(1-2\alpha)c/4} + \cdots)\\
&< c/(e^{(1-2\alpha)c/4} - 1) < 4/(1 - 2\alpha)
\end{align*}
so $I_n \ge A^nI_0e^{-2/(1 - 2\alpha)} \ge c$ anyway.
\end{proof}

\begin{lem}\label{int-le2}
Let $\alpha \in (0, 1/2)$ and $A > 1$.
Let $I_0(t) = t$ and $I_{n+1}(t) = A(1 - e^{-I_n(t)})$.
Let $x_n(t) = A^ne^{-I_0(t) - \dots - I_{n-1}(t)}$. Then
\[
\int_0^{\frac{4e^{2/(1-2\alpha)}\ln A}{3 - 2\alpha}} x_n(t)^{1-\alpha}A^{\alpha n}dt
< A\left( 1 + \frac{A - 1}{1 - A^{(2\alpha-1)/(3-2\alpha)}} \right)\frac{4e^{2/(1-2\alpha)}\ln A}{3 - 2\alpha}.
\]
\end{lem}
\begin{proof}
For $k = 0, \dots, n - 1$ let
\[
t_k = \frac{4e^{2/(1-2\alpha)}\ln A}{3 - 2\alpha}A^{-k}.
\]
By Lemma \ref{int-ge2},
\[
I_k \ge \frac{4\ln A}{3 - 2\alpha}
\]
so $x_n(t) \le A^{n-4(n-k)/(3-2\alpha)}$ and $x_n(t)^{1-\alpha}A^{\alpha n} \le A^{n-4(1-\alpha)(n-k)/(3-2\alpha)}$. Since the upper limit of the integral is $t_0$,
\begin{align*}
\text{the integral}
&\le A^nt_{n-1} + \sum_{k=1}^{n-1} A^{n-4(1-\alpha)(n-k)/(3-2\alpha)}(t_{k-1} - t_k)\\
&= \left( A + \sum_{k=1}^{n-1} (A - 1)(A^{(2\alpha-1)/(3-2\alpha)})^{n-k} \right)\frac{4e^{2/(1-2\alpha)}\ln A}{3 - 2\alpha}\\
&< A\left( 1 + \frac{A - 1}{1 - A^{(2\alpha-1)/(3-2\alpha)}} \right)\frac{4e^{2/(1-2\alpha)}\ln A}{3 - 2\alpha}.
\end{align*}
\end{proof}

\begin{proof}[Closing the last two bootstrap assumptions]
Note $X_{n,k}(t) = x_{n,k}(t)^{5/4}/A^{k/4}$ satisfies
\[
\frac{\dot X_{n,k}(t)}{X_{n,k}(t)} = \frac54\sum_{j<k} a_{n,j}(t) X_{n,j}(t).
\]
where $|a_{n,k}(t) - 1| \le 3c(d, Z)/2$. By Lemma \ref{int-ge},
\[
\frac{5(2 + 3c(d, Z))}{8}\int_{-T}^0 X_{n,k}(s)ds
\]
has a lower bound satisfying the recursion in Lemma \ref{int-le2}, with
\[
T = \min\left( a(d, Z, A), \frac{32e^{2/(1-2\alpha)}\ln A}{5(3-2\alpha)(2+3c(d,Z))} \right),
\]
Let $\alpha = 7/15 < 1/2$. Then $X_{n,k}(t)^{1-\alpha}A^{\alpha k} = x_{n,k}(t)^{2/3}A^{k/3}$
and Lemma \ref{int-le2} gives
\[
\int_{-T}^0 x_{n,k}(s)^{2/3}A^{k/3}ds < \left( 1 + \frac{A - 1}{1 - A^{-1/31}} \right)\frac{96Ae^{30}\ln A}{31(2 + 3c(d, Z))}.
\]
As $A \to 1+$, $(A - 1)/(1 - A^{-1/31}) \to 31$, $\ln A \to 0$ and $C(d, Z, A)$ remains bounded,
so if $A$ is sufficiently close to 1,
\[
C(d, Z, A)\int_{-T}^0 x_{n,k}(s)^{2/3}A^{k/3}ds < 1 + d
\]
and the bootstrap assumptions are closed.
\end{proof}

\subsection{Smoothing out at $z = 0$}\label{3Dsmooth}
We have constructed a solution
\[
w_n(r, z, t) = \sum_{k=0}^n w_{n,k}(r, z, t)
\]
to the 3D axi-symmetric Euler equation on the time interval $[-T, 0]$ for some $T > 0$ independent of $n$,
with
\[
w_n(r, z, 0) = \sum_{k=0}^n A^k\phi(r/d^k)\rho(z/d^k)|z/d^k|^{1/12}\sgn z
\]
where $\phi$ is a fixed smooth function and $\rho(z) = 1$ for $|z| \le Z$.
Now we let $\rho_k(z)$ be $\rho(z)|z|^{1/12}\sgn z$ suitably smoothed near 0 and specify $w$ at time 0 to be instead
\[
w_n(r, z, 0) = \sum_{k=0}^n A^k\phi(r/d^k)\rho_k(z/d^k).
\]
The only properties of $\rho$ used in the proof are that it is supported in $[-2Z, 2Z]$
(which is preserved under smoothing near 0),
\[
\left| \int_0^\infty \int_0^\infty \frac{r^2z\phi(r)\rho(Kz)(Kz)^{1/12}}{(z^2 + r^2)^{5/2}}drdz
- \frac{2K^{1/12}}{3} \right| \le \frac{32K^{1/12}(1 + d)^{35/6}}{35BZ^{35/12}}.
\]
for $K \in [0, 1]$ (used in the proof of Lemma \ref{HW-1})
and the finiteness of the $\dot C^{1/12}$ norm of
\[
\phi(r)\rho(z)|z|^{1/12}\sgn z/r
\]
(used in Lemma \ref{w-Holder}). After the smoothing,
The $\dot C^{1/12}$ norm of $\phi(r)\rho_k(z)/r$ is bounded uniformly in $k$,
provided that both the $C^0$ norm and the $\dot C^{1/12}$ norm of $\rho_k$
are bounded uniformly in $k$, which is easy to satisfy.
For the integral we restrict the smoothing to $[-h_k, h_k]$, where $0 < h_k < Z$. Then
\begin{align*}
\iint \frac{r^2z\phi(r)|\rho(Kz)(Kz)^{1/12} - \rho_k(Kz)|}
{(z^2 + r^2)^{5/2}}drdz
&\le K^{1/12}\int_0^{h_k/K} z^{13/12}dz\int \frac{\phi(r)}{r^3}dr\\
&= \frac{16h_k^{25/12}}{25K^2B}.
\end{align*}
Since we used Lemma \ref{HW-1} with $K = x_{n,k}(t)^3/A^{3k} > e^{-6ak}$,
we can let $h_k = \epsilon e^{-6ak}$ to ensure that the right-hand side is less than $\epsilon^{25/12}K^{1/12}/B$,
which can be arbitrarily small compared to $K^{1/12}$ if $\epsilon$ is sufficiently small.
Then the proof goes through as before.

\subsubsection{Higher H\"older estimates}
Since $\rho_k$ is smooth, so are $w_n(r, z, 0)$ and $w_n(r, z, t)$
because solutions to the 3D axi-symmetric Euler equation are global.
We now estimate $\dot C^{N,1/12}$ norm ($N \ge 1$) of $w_{n,k}$, uniformly in $n$.

Recall that $\tilde w_{n,k}(r, z, t) = w_{n,k}(r, z, t)/r$, so
\[
\|w_{n,k}(\cdot, \cdot, t)\|_{\dot C^{N,1/12}}
\le C_N\|\tilde w_{n,k}(\cdot, \cdot, t)\|_{C^{N,1/12}}\|r\|_{C^{N,1/12}}
\]
where all the norms are taken on the support of $w_{n,k}(\cdot, \cdot, t)$, where
\[
(1 - d)^2(de^{-2a})^k \le (1 - d)^2x_{n,k}(t)(d/A)^k \le  r \le (1 + d)^2x_{n,k}(t)(d/A)^k \le (1 + d)^2d^k
\]
and
\[
|z| \le 2Z(1 + d)(A^2d)^k/x_{n,k}(t)^2 \le 2Z(1 + d)(e^{4a}d)^k.
\]
Thus $w_{n,k}(t)$ is supported in a set depending only on $k$, $d$, $Z$ and $A$ and bounded away from the $z$ axis. Then
\[
\|r\|_{C^{N,1/12}}\text{ and }\|\tilde w_{n,k}(\cdot, \cdot, t)\|_{C^{N,1/12}}/\|\tilde w_{n,k}(\cdot, \cdot, t)\|_{\dot C^{N,1/12}}
\]
are bounded on a constant depending only on $N$, $k$, $d$, $Z$ and $A$, so is
\[
\|w_{n,k}(\cdot, \cdot, t)\|_{C^{N,1/12}}/\|\tilde w_{n,k}(\cdot, \cdot, t)\|_{\dot C^{N,1/12}}
\]
and it suffices to bound the $\dot C^{N,1/12}$ norm of $\tilde w_{n,k}$.

Since $\tilde w_{n,k}$ is transported by the flow,
$(\partial_t + u_n^r\partial_r + u_n^z\partial_z)\tilde w_{n,k} = 0$.
Taking the $N$-th derivative gives
\[
(\partial_t + u_n^r\partial_r + u_n^z\partial_z)\nabla^N\tilde w_{n,k}
= \sum_{m=1}^N (\nabla^mu_n^r\partial_r + \nabla^mu_n^z\partial_z)\nabla^{N-m}\tilde w_{n,k}.
\]
We use the following bilinear H\"older estimate:
\begin{lem}\label{Holder-bilin}
For smooth functions $f$ and $g$,
\begin{align*}
\|\nabla^mf\nabla^{m'}g\|_{\dot C^{1/12}}
&\le \max|\nabla^mf|\|\nabla^{m'}g\|_{\dot C^{1/12}}
+ \|\nabla^mf\|_{\dot C^{1/12}}\max|\nabla^{m'}g|\\
&\le C_{m,m'}(\max|f|\|g\|_{\dot C^{m+m',1/12}})^{\frac{m'+1/12}{m+m'+1/12}}\\
&\times (\|f\|_{\dot C^{m+m',1/12}}\max|g|)^{\frac{m}{m+m'+1/12}}\\
&+ C_{m,m'}(\|f\|_{\dot C^{m+m',1/12}}\max|g|)^{\frac{m+1/12}{m+m'+1/12}}\\
&\times (\max|f|\|g\|_{\dot C^{m+m',1/12}})^{\frac{m'}{m+m'+1/12}}\\
&\le C_{m,m'}(\max|f|\|g\|_{\dot C^{m+m',1/12}} + \|f\|_{\dot C^{m+m',1/12}}\max|g|).
\end{align*}
\end{lem}

\begin{lem}\label{gru-CN}
If $d \le 1/4$, $r \le (1 + d)^2x_{n,k}(t)(d/A)^k$, $|z| \le 2Z(A^2d)^k/x_{n,k}(t)^2$
then for $t \in [-T, 0]$ and $N \ge 1$,
\[
|\nabla u_\pm(\cdot, \cdot, t)|_{\dot C^{N,1/12}}
\le C_\pm(N, d, Z, A)x_{n,k}(t)^{5/6-N}A^{(N+1/6)k}d^{-(N+1/12)k}
\]
which is in turn bounded by a constant only depending on $N$, $k$, $d$, $Z$ and $A$.
\end{lem}
\begin{proof}
As in Lemma \ref{gru--Holder},
\[
|\nabla u_-(r, z, t)|_{\dot C^{N,1/12}} \le C_N\iiint \frac{\|w_{n,j}(\cdot, \cdot, t)\|_{\dot C^{1/12}}}{|y|^{N+3}}dy
\]
where the $\dot C^{1/12}$ norm of $w_{n,j}$ is given by Lemma \ref{w-Holder}.

For $j < k$, the integral is on the region
\[
\{1 - 3d - d^2 - d^3 \le r(A/d)^j/x_{n,j}(t) \le (1 + d)^3\}
\]
so
\begin{align*}
\iiint \frac{dy}{|y|^{N+3}}dy
&= (A/d)^{Nj}x_{n,j}(t)^{-N} C_N\int_{1-3d-d^2-d^3}^{(1+d)^3} \frac{dr}{r^{N+1}}\\
&= C(N, d)(A/d)^{Nj}x_{n,j}(t)^{-N}.
\end{align*}
Summing over $j < k$ shows the bound for $u_-$.

For $j > k$,
\begin{align*}
\iiint \frac{dy}{|y|^3}
&\le \int_{-\infty}^\infty \frac{\pi(1 + d)^6x_{n,j}(t)^2(d/A)^{2j}dz}{((1 - 3d - d^2 - d^3)^2x_{n,k}(t)^2(d/A)^{2k} + z^2)^{(N+3)/2}}\\
&= \frac{C_N(1 + d)^6x_{n,j}(t)^2(d/A)^{2j-(N+2)k}}{(1 - 3d - d^2 - d^3)^2x_{n,k}(t)^{N+2}}\\
&\le \frac{C_N(1 + d)^6d^{2j-(N+2)k}A^{Nk}}{(1 - 3d - d^2 - d^3)^2x_{n,k}(t)^N}
\end{align*}
Summing over $j > k$ shows the bound for $u_+$.

The last sentence follows from $(Ae^{-2a})^k \le x_{n,k}(t) \le A^k$.
\end{proof}

\begin{lem}\label{w-Coo}
$\tilde w_{n,k}(\cdot, \cdot, t)$ is bounded in $C^\infty$, uniformly in $n$.
\end{lem}
\begin{proof}
Since we already have a uniform $\dot C^{1/12}$ norm for $\tilde w_{n,k}$
(which, together with a uniform bound on its support, gives a uniform bound on its sup norm),
it suffices to bound, for each $N \ge 1$, the $\dot C^{N,1/12}$ norm uniformly in $n$.

Denote by $C$ a constant depending only on $N$, $k$, $d$, $Z$ and $A$, but not on $n$. By Lemma \ref{Holder-bilin},
each summand $(\nabla^mu_n^r\partial_r + \nabla^mu_n^z\partial_z)\nabla^{N-m}\tilde w_{n,k}$ ($m = 1, \dots, N$)
and hence the sum $(\partial_t + u_n^r\partial_r + u_n^z\partial_z)\nabla^N\tilde w_{n,k}$ has $\dot C^{1/12}$ norm bounded by
\[
C_N(\max|\nabla u_n|\|\tilde w_{n,k}\|_{\dot C^{N,1/12}} + \|\nabla u_n\|_{\dot C^{N,1/12}}\max|\tilde w_{n,k}|).
\]
Note that
\[
\max|\tilde w_{n,k}(\cdot, \cdot, t)| = \max|\tilde w_{n,k}(\cdot, \cdot, 0)|
= (A/d)^k(2Z)^{1/12}\max|\phi(r)/r| \le C
\]
and
\begin{align*}
\max|\nabla u_n(\cdot, \cdot, t)|
&\le 2\dot x_{n,k}(t)/x_{n,k}(t) + C(d, Z, A)x_{n,k}(t)^{2/3}A^{k/3}\\
&\le 2\dot x_{n,k}(t)/x_{n,k}(t) + C.
\end{align*}
By Lemma \ref{gru-CN},
\begin{align*}
\|\nabla u_n\|_{\dot C^{N,1/12}}
&\le C + \|\nabla u_{n,k}\|_{\dot C^{N,1/12}} \le C + C_N\|w_{n,k}\|_{\dot C^{N,1/12}}\\
&\le C + C\|\tilde w_{n,k}\|_{\dot C^{N,1/12}}
\end{align*}
so
\[
\|(\partial_t + u_n^r\partial_r + u_n^z\partial_z)\nabla^N\tilde w_{n,k}\|_{\dot C^{1/12}}
\le (2\dot x_{n,k}(t)/x_{n,k}(t) + C)\|\tilde w_{n,k}\|_{\dot C^{N,1/12}} + C.
\]
Integration along the characteristics shows that
\begin{align*}
\nabla^N\tilde w_{n,k}(r, z, t) &= \nabla^N\tilde w_{n,k}(Y(r, z, 0, t), 0)\\
&- \int_t^0 (\partial_t + u_n^r\partial_r + u_n^z\partial_z)\nabla^N\tilde w_{n,k}(Y(r, z, s, t), s)ds.
\end{align*}
By bootstrap assumption 3, $|DY(\cdot, \cdot, s, t)| \le 2A^{2k}/x_{n,k}(t)^2 \le 2e^{4ak} \le C$, so
\begin{align*}
\|\nabla^N\tilde w_{n,k}(Y(r, z, 0, t), 0)\|_{\dot C^{1/12}}
&\le C\|\nabla^N\tilde w_{n,k}(r, z, 0)\|_{\dot C^{1/12}}\\
&= CA^k\|\phi(r/d^k)\rho_k(z/d^k)\|_{\dot C^{N,1/12}}\le C,\\
\|\text{the integrand}\|_{\dot C^{1/12}}
&\le C(\dot x_{n,k}(t)/x_{n,k}(t) + 1)\|\tilde w_{n,k}\|_{\dot C^{N,1/12}} + C.
\end{align*}
Then Gronwall's inequality shows that
\[
\|\tilde w_{n,k}(\cdot, \cdot, t)\|_{\dot C^{N,1/12}}
\le e^{C(x_{n,k}(0)/x_{n,k}(t)+t)}C(1 + |t|)
\le e^{C(e^{2ak}+T)}C(1 + T)
\]
which can be bounded uniformly in $n$.
\end{proof}

\subsection{Convergence to the limit}\label{3Dlimit}
We have constructed a solution
\[
w_n(r, z, t) = \sum_{k=0}^n w_{n,k}(r, z, t)
\]
to the 3D axi-symmetric Euler equation on the time interval $[-T, 0]$ for some $T > 0$ independent of $n$,
with
\[
w_n(r, z, 0) = \sum_{k=0}^n A^k\phi(r)\rho_k(z)
\]
where $\phi$ and $\rho_k$ are smooth. Now we show that $w_n$ converges to a solution as $n \to \infty$.

By Lemma \ref{w-Coo}, $w_{n,k}(\cdot, \cdot, t)$ is bounded in $C^\infty$, uniformly in $n$. By a diagonalization argument, we can pass to a subsequence such that for each $n$, $w_{n,k}(\cdot, \cdot, t) \to w_k^*(\cdot, \cdot, t)$ uniformly in $C^\infty$.
Since on $\supp w_{n,k}(\cdot, \cdot, t)$, $r \ge (1 - d)^2x_{n,k}(t)(d/A)^k \ge (1 - d)^2e^{-2ak}d^k$,
the convergence is also in $C^\infty$ if viewed as functions on $\R^3$.
Also $r \le (1 + d)^2x_{n,k}(t)(d/A)^k \le (1 + d)^2d^k$ and $|Z| \le 2Z(1 + d)(A^2d)^k/x_{n,k}(t)^2 \le 2Z(1 + d)^2(e^{4a}d)^k$. Since $d \le 1/4$ and $A - 1$ (and hence $a$) is sufficiently small,
$\supp w_{n,k}(\cdot, \cdot, t)$ converges to the origin, i.e.,
any point other than the origin has a neighborhood which intersects
only finitely many $\supp w_{n,k}(\cdot, \cdot, t)$. Then for all $N$,
\[
w_n \to w = \sum_{k=0}^\infty w_k^*
\]
in $C^N$ on $\{r + |z| \ge 1/N\}$. Since
\begin{align*}
\max_x|w_{n,k}(x, t)|
&\le (1 + d)^2x_{n,k}(t)(d/A)^k\max|\tilde w_{n,k}(\cdot, \cdot, t)|\\
&\le (1 + d)^2(2Z)^{1/12}A^k\max|\phi(r)/r|
\end{align*}
and $|\supp_x w_{n,k}(x, t)| = |\supp_x w_{n,k}(x, 0)| \le 8\pi d^{3k+1}Z$,
\[
\|w_{n,k}(x, t)\|_{L^6_x} \le 2(1 + d)^2(A\sqrt d)^k(2Z)^{1/4}\max|\phi(r)/r|.
\]
Since $d \le 1/4$ and $A$ is sufficiently close to 1,
$\sum_k \|w_{n,k}(x, t)\|_{L^3_x}$ converges, uniformly in $n$.
Then by spatial localization, $u_n \to u$ in $C^N$ on $\{r + |z| \ge 2/N\}$,
with $w = \nabla \times u$ there, so $u$ is a smooth solution to the Euler equation away from the origin.
Also, by singular integral estimates and Morrey's inequality,
$\sum_k \|\nabla_x u_{n,k}(x, t)\|_{L^6_x}$ and $\sum_k \|u_{n,k}(x, t)\|_{\dot C^{1/2}_x}$
converge, uniformly in $n$, so $u_n \to u$ in $\dot C^{1/2}$ on the whole space.
Since $u_n(0, 0, 0) = 0$ and $u_n \to u$ in $C^1$ on $\{r + |z| \ge 2\}$,
the convergence is in $C^{1/2}$ on the whole space.
Since $w$ is integrable, $u$ decays like $1/|x|^2$ as $|x| \to \infty$,
so $u \in L^2$. Also $u$ is smoother than the $C^{1/3}$ threshold,
so $u$ is weak solution to the 3D Euler solution with finite, conserved energy,
that is smooth away from the origin.

\subsection{Properties of the solution}\label{3Dprop}
Now we show that our solution blows up in the H\"older norm.

\subsubsection{H\"older continuity}
\begin{thm}
If $|t|$ is small enough, the vorticity vector $we_\theta$ and the deformation tensor $\nabla u$ are H\"older continuous.
\end{thm}
\begin{proof}
If two points $x$ and $y$ are in the supports of $w_{n,j}$ and $w_{n,k}$ with $j \neq k$,
then we can find $z$ in the segment joining $x$ and $y$ such that $w(z, t) = 0$.
Thus it suffices to show the H\"older continuity of each single $w_{n,k}$ uniformly in $n$ and $k$.
From the second bound of Lemma \ref{w-Holder} we see that it suffices to show the H\"older continuity of the scalar vorticity $w_{n,k}(\cdot, \cdot, t)$ uniformly in $n$ and $k$.
Recall that $w_{n,k}(r, z, t) = r\tilde w_{n,k}(r, z, t)$, so
\begin{align*}
\max|\tilde w_{n,k}(\cdot, \cdot, t)|
&= (A/d)^k(2Z)^{1/12}\max|\phi(r)/r|,\\
\|\tilde w_{n,k}(\cdot, \cdot, t)\|_{\dot C^s}
&\le \max|DY(\cdot, \cdot, 0, t)|^s\|\tilde w_{n,k}(\cdot, \cdot, 0)\|_{\dot C^s}\\
&\le (2A^{2k}/x_{n,k}(t)^2)^sA^kd^{-(1+s)k}\|\phi(r)\rho_k(z)/r\|_{\dot C^s}.
\end{align*}
Then
\begin{align*}
\|w_{n,k}(\cdot, \cdot, t)\|_{\dot C^s}
&\le \|\tilde w_{n,k}(\cdot, \cdot, t)\|_{\dot C^s}\max_{\supp w_{n,k}}r\\
&+ \max|\tilde w_{n,k}(\cdot, \cdot, t)|\|r\|_{\dot C^s(\supp w_{n,k})}\\
&\le (2A^{2k}/x_{n,k}(t)^2)^sA^kd^{-(1+s)k}(1 + d)^2x_{n,k}(t)d^k/A^k\\
&\times \|\phi(r)\rho_k(z)/r\|_{\dot C^s}\\
&+ (A/d)^k(2Z)^{1/12}((1 + d)^2x_{n,k}(t)d^k/A^k)^{1-s}\max|\phi(r)/r|\\
&= (2A^{2k}/x_{n,k}(t)^2)^s(1 + d)^2x_{n,k}(t)d^{-sk}\|\phi(r)\rho_k(z)/r\|_{\dot C^s}\\
&+ (2Z)^{1/12}(1 + d)^{2-2s}x_{n,k}(t)(A^k/x_{n,k}(t))^sd^{-sk}\max|\phi(r)/r|\\
&\le Cx_{n,k}(t)^{1-2s}A^{2sk}d^{-sk}
\end{align*}
where
\[
C = 2^s(1 + d)^2\|\phi(r)\rho_k(z)/r\|_{\dot C^s} + (2Z)^{1/12}(1 + d)^{2-2s}\max|\phi(r)/r|.
\]
In the smoothing process, we arranged that the $\dot C^{1/12}$ norm of $\phi(r)\rho_k(z)/r$ is bounded uniformly in $k$,
so is its $\dot C^s$ norm for any $s \le 1/12$ because its support is bounded uniformly in $k$.
Then $C$ has a bound depending only on $s$, $d$ and $Z$. Then it suffices that
\[
1 > \limsup_{k\to\infty} x_{n,k}(t)^{(1-2s)/k}(A^2/d)^s \ge (Ae^{-a})^{1-2s}(A^2/d)^s
\]
where $a = A(1 - e^{-a}) > \ln A$,
then the above inequality holds for if $s > 0$ is sufficiently small,
so $w_{n,k}(\cdot, \cdot, t)$ is H\"older continuous, uniformly in $n$ and $k$, as desired.
\end{proof}

\subsubsection{Rate of the blow up}
\begin{thm}
There is a constant $C > 0$ such that for all $T \le t < 0$, $1/|Ct| \le \max|w(\cdot, t)| \le C/|t|$.
\end{thm}
\begin{proof}
Since
\begin{align*}
\max|w_{n,k}(\cdot, \cdot, t)|
&\le (1 + d)^2x_{n,k}(t)(d/A)^k\max|\tilde w_{n,k}(\cdot, \cdot, t)|\\
&= (1 + d)^2(2Z)^{1/12}x_{n,k}(t)\max|\phi(r)/r|,
\end{align*}
it suffices to show the same bound for $\sup_k x_{n,k}(t)$, uniformly in $n$.

For the upper bound, since $x_{n,k}(t)$ is increasing in $t$ to $A^k$, and
\[
\int_{-T}^0 x_{n,k}(s)^{2/3}A^{k/3}ds
\]
is bounded uniformly in $n$ and $k$, by the last estimate in section \ref{3DHolder},
\[
x_{n,k}(t) \le \frac{1}{|t|}\int_T^0 x_{n,k}(s)ds 
\le \frac{1}{|t|}\int_T^0 x_{n,k}(s)^{2/3}A^{k/3}ds \le \frac{C}{|t|}.
\]

For the lower bound, when $-A^{-k} \le t < -A^{-(k+1)}$, the proof in the De Gregorio case shows that
\(
x_{n,k}(t) \ge x_{n,k}(t)^{5/4}/A^{k/4} = X_{n,k}(t) \ge 1/(C|t|).
\)
\end{proof}

\subsubsection{Non-asymptotic self-similarity}\label{3DNonsim}
The authors have yet to locate an exact definition of asymptotic self-similarity for solutions to the 3D axi-symmetric Euler equation, so they have not set themselves the task of proving non-asymptotic self-similarity rigorously. Nonetheless, considering that the vorticity $w$ of the solution is summed up from infinitely compactly supported pieces, separated by vortex-free regions, the authors believe that once a precise definition of asymptotic self-similarity is available, it is not hard to show that our example of blow up is not asymptotically similar, just as we did for the De Gregorio equation.

\section*{Acknowledgements}
This work is supported in part by the Spanish Ministry of Science
and Innovation, through the “Severo Ochoa Programme for Centres of Excellence in R$\&$D (CEX2019-000904-S)” and 114703GB-100. We were also partially supported by the ERC Advanced Grant 788250. 

\bibliographystyle{alpha}

\end{document}